\documentclass[nosumlimits,twoside]{amsart}

\usepackage{amsfonts, amsmath, amssymb}
\usepackage[bookmarksnumbered,plainpages,hypertex]{hyperref}
\usepackage{graphicx}
\usepackage{float}
\usepackage{srcltx}
\usepackage[usenames]{color}

\newtheorem{theorem}{\sc Theorem}[section]
\newtheorem{proposition}[theorem]{\sc Proposition}

\newtheorem{lemma}[theorem]{\sc Lemma}
\newtheorem{corollary}[theorem]{\sc Corollary}
\theoremstyle{definition}
\newtheorem{definition}[theorem]{\sc Definition}

\newtheorem{example}[theorem]{\sc Example}

\theoremstyle{remark}
\newtheorem{remark}[theorem]{\sc Remark}
\newtheorem{remarks}[theorem]{\sc Remarks}

\def\x1{x_1}
\def\x2{x_2}
\def\a1{a_1}
\def\a2{a_2}

\begin{document}
\title[Cocycle deformations]{Cocycle deformations for Hopf algebras with a
coalgebra projection}

\author{Alessandro Ardizzoni}
\address{Department of Mathematics, University of Ferrara, Via Machiavelli
35, Ferrara I-44121, Italy} \email{alessandro.ardizzoni@unife.it}
\urladdr{http://www.unife.it/utenti/alessandro.ardizzoni}

\author{Margaret Beattie}
\address{Department of Mathematics and Computer Science, Mount Allison
University, Sackville, NB E4L 1E6, Canada} \email{mbeattie@mta.ca}
\urladdr{http://www.mta.ca/~mbeattie/}

\author{Claudia Menini}
\address{Department of Mathematics, University of Ferrara, Via Machiavelli
35, Ferrara I-44121, Italy} \email{men@dns.unife.it}
\urladdr{http://www.unife.it/utenti/claudia.menini}

\thanks{M. Beattie's research was supported by an NSERC Discovery Grant.
Thanks to the University of Ferrara and to Mount Allison
University for the warm hospitality extended to M. Beattie and to
A. Ardizzoni during their visits of 2007 and 2008 respectively.
This paper was written while A. Ardizzoni and C. Menini were
members of G.N.S.A.G.A. with partial financial support from
M.I.U.R. (PRIN 2007).}


\begin{abstract}
Let $H$ be a Hopf algebra over a field $K$ of characteristic $0$
and let $A$ be a bialgebra or Hopf algebra such that $H$ is
isomorphic to a sub-Hopf algebra of $A$ and there is an
$H$-bilinear coalgebra projection $\pi$ from $A$ to $H$ which
splits the inclusion. Then $A \cong R \#_\xi H$ where $R$ is the
pre-bialgebra of coinvariants. In this paper we study the
deformations of $A$ by an $H$-bilinear cocycle. If $\gamma$ is a
cocycle for $A$, then $\gamma$ can be restricted to a cocycle
$\gamma_R$ for $R$, and $A^\gamma \cong R^{\gamma_R}
\#_{\xi_\gamma} H$. As examples, we consider liftings of
$\mathcal{B}(V) \# K[\Gamma]$ where $\Gamma$ is a finite abelian
group, $V$ is a quantum plane and $\mathcal{B}(V)$ is its Nichols
algebra, and explicitly construct the cocycle which twists the
Radford biproduct into the lifting.
\end{abstract}


\maketitle \tableofcontents

\section{Introduction}

Let $A$, $H$ be Hopf algebras over a field $K$ of characteristic 0 and
suppose that $\sigma :H\hookrightarrow A$ embeds $H$ as a sub-Hopf algebra
of $A$. If there is a Hopf algebra projection $\pi :A\rightarrow H$ such
that $\pi \circ \sigma $ is the identity, then $A$ is isomorphic to a
Radford biproduct $R\#H$ \cite{Rad} of the algebra of co-invariants $%
R=A^{co\pi }$ and the Hopf algebra $H$. In this setting $R$ is not a
sub-Hopf algebra of $A$ but is a Hopf algebra in the Yetter-Drinfeld
category $_{H}^{H}\mathcal{YD}$.

Suppose $H$ is a Hopf algebra over $K$, $A$ a bialgebra, $\sigma: H
\hookrightarrow A$ a bialgebra embedding and $\pi$ an $H$-bilinear coalgebra
homomorphism from $A$ to $H$ that splits $\sigma$. Then the 4-tuple $%
(A,H,\pi,\sigma)$ is called a splitting datum. If $\pi$ is an algebra
homomorphism, then $A$ is a Radford biproduct as above. More generally, $A =
R \#_\xi H$, where $R$ is the set of $\pi$-coinvariants and is a coalgebra
in $_{H}^{H}\mathcal{YD}$ which is not a bialgebra but what was termed in
\cite{A.M.S.} a pre-bialgebra with cocycle $\xi$. If $\pi$ is only left $H$%
-linear, then $\pi$ was called a weak projection by Schauenburg \cite{Scha};
he showed that bicrossproducts, double crossproducts and all quantized
universal enveloping algebras are examples of this situation.

Using the machinery from \cite{A.M.S.}, \cite{A.M.St.-Small} and \cite%
{AM-Small2}, we explore the relationship between the associated
pre-bialgebras for a splitting datum $(A,H,\pi,\sigma)$ and the splitting
datum $(A^\gamma, H,\pi,\sigma)$ where $A^\gamma$ is a cocycle deformation
of $A$. Cocycle deformations of Hopf algebras are of interest in the problem
of the classification of Hopf algebras. In particular, it has recently been
proved (see \cite[Theorem 4.3]{Grunenfelder-Mastnak} and \cite[Appendix]%
{Mas2}) that the families of finite-dimensional pointed Hopf algebras with
the same associated graded Hopf algebra $\mathcal{B}(V) \# K[\Gamma]$
classified by Andruskiewitsch and Schneider in \cite{AS2} are cocycle
deformations of a Radford biproduct. We define the notion of a cocycle twist
for a pre-bialgebra $(R,\xi)$ and show that given a splitting datum as
above, if $A^\gamma$ is a cocycle twist of $A$ then $A^\gamma \cong
R^{\gamma_R} \#_{\xi_{\gamma}} H$ where $R^{\gamma_R}$ is a cocycle twist of
$R$.

This paper is organized in the following way. In a preliminary section we
first recall basic facts about coalgebras in the Yetter-Drinfeld category $%
_{H}^{H} \mathcal{YD}$, prove some key lemmas, and review the basic theory
of pre-bialgebras with a cocycle in $_{H}^{H} \mathcal{YD}$ from \cite%
{A.M.S.} and \cite{A.M.St.-Small}, ending with some examples. In general a
pre-bialgebra with cocycle $(R,\xi)$ does not have associative
multiplication for $\xi$ nontrivial. In Section \ref{sec: xi} we show that
if $R$ is connected, the sufficient conditions for $(R,\xi)$ to have
associative multiplication from Section \ref{sec: pre-bialg} are also
necessary. Section \ref{sec: main} contains the main results of this paper.
Here we review the notion of a cocycle twist for a Hopf algebra $A$, and
define cocycle twists for pre-bialgebras. We show that for $(R,\xi)$ a
pre-bialgebra with cocycle associated to a splitting datum $(A,H,\pi,
\sigma) $, then the set of $H$-bilinear cocycles on $A$ is in one-one
correspondence with the left $H$-linear cocycles on $R$. Furthermore, a
cocycle twist of $A$, say $A^\gamma \cong (R \# _\xi H)^\gamma $, is
isomorphic to $R^{\gamma_R} \#_{\xi_\gamma} H$ where $\gamma_R$ is the
cocycle on $R \otimes R$ corresponding to the cocycle $\gamma$ for $A$ and $%
(R^{\gamma_R}, \xi_\gamma) $ is the pre-bialgebra with cocycle corresponding
to $A^\gamma$. In Section \ref{sec: qlps}, we explicitly describe the
cocycle which twists the Radford product $\mathcal{B}(V) \# K[\Gamma]$ of
the group algebra of a finite abelian group and the Nichols algebra of a
quantum plane to the liftings of this pointed Hopf algebra. Examples include
the three families of non-isomorphic pointed Hopf algebras of dimension $32$
described in \cite{grana} and the pointed Hopf algebras of dimension $81$
which were among the first counterexamples to Kaplansky's Tenth Conjecture.

Throughout $H$ will denote a Hopf algebra over a field $K$. All maps are
assumed to be over $K$. We assume for simplicity of the exposition that our
ground field $K$ has characteristic zero. Anyway we point out that many
results below are valid under weaker hypotheses.

\section{Preliminaries}

We will use the Heyneman-Sweedler notation for the comultiplication in a $K$%
-coalgebra $C$ but with the summation sign omitted, namely $\Delta
(x)=x_{(1)}\otimes x_{(2)}$ for $x \in C$. For $C$ a coalgebra and $A$ an
algebra the convolution multiplication in $\mathrm{Hom}(C,A)$ will be
denoted $*$. Composition of functions will be denoted by $\circ$ or possibly
by juxtaposition when the meaning is clear.

A Hopf algebra $H$ is a left $H$-module under the adjoint action $h
\rightharpoonup m = h_{(1)}m S(h_{(2)})$ and has a similar right adjoint
action. Recall \cite[Definition 2.7]{A.M.S.} that a left and right integral $%
\lambda \in H^*$ for $H$ is called ad-invariant if $\lambda(1) = 1$ and $%
\lambda$ is a left and right $H$-module map with respect to the left and
right adjoint actions. If $H$ is semisimple and cosemisimple, then the total
integral for $H$ is ad-invariant; see either \cite[Proposition 1.12, b)]{SvO}
or \cite[Theorem 2.27]{A.M.S.}. If $H$ has an ad-invariant integral, then $H$
is cosemisimple.

We assume familiarity with the general theory of Hopf algebras; good
references are \cite{Sw}, \cite{Mo}.

\subsection{Coalgebras in a category of Yetter-Drinfeld modules}

\label{sec: ydcoalgebras}

Let $H$ be a Hopf algebra over $K$. Coalgebras in $_{H}^{H} \mathcal{YD}$,
the category of left-left Yetter-Drinfeld modules over $H$, will play a
central role in this paper. For $(V, \cdot)$ a left $H$-module, we write $hv$
for $h \cdot v$, the action of $h$ on $v$, if the meaning is clear. The left
$H$-module $H$ with the left adjoint action is denoted $(H,\rightharpoonup)$%
; the left and right actions of $H$ on $H$ induced by multiplication will be
denoted by juxtaposition. For $(V, \rho)$ a left $H$-comodule, for $v \in V$
we write $\rho(v) = v_{\langle -1 \rangle} \otimes v_{\langle 0 \rangle}$
for the coaction. Recall that if $V$ is a left $H$-module and a left $H$%
-comodule, then $V$ is an object in $_{H}^{H} \mathcal{YD}$ if for all $v
\in V$, $h \in H$,
\begin{equation*}
\rho(h\cdot v) = h_{(1)}v_{\langle -1 \rangle}S(h_{(3)}) \otimes h_{(2)}
\cdot v_{\langle 0 \rangle}.
\end{equation*}
The field $K$ is a left-left Yetter-Drinfeld module with $\rho(1) = 1
\otimes 1$ and trivial left $H$ action. As well, $(H, \rightharpoonup,
\Delta_H)$ is a left-left Yetter-Drinfeld module. If $V,W$ are objects in $%
_{H}^{H} \mathcal{YD}$, so is $V \otimes W$ with $H$ action given by $h(r
\otimes t) = h_{(1)}r \otimes h_{(2)}t$, and $H$-coaction given by $\rho(r
\otimes t) = r_{\langle -1 \rangle}t_{\langle -1 \rangle} \otimes r_{\langle
0 \rangle} \otimes t_{\langle 0 \rangle}$ for all $r \in V$, $t \in W$, $h
\in H$. A map $f \in \mathrm{Hom}(V,W)$ is called (left) $H$-linear if $f(h
\cdot v) = h \cdot f(v)$, for all $h\in H,v\in V$. For example $f \in
\mathrm{Hom}(V,K)$ is left $H$-linear if $f(h \cdot v ) = \varepsilon(h)f(v)$
and $f \in \mathrm{Hom}(V,H)$ is left $H$-linear if $f(h \cdot v) =
h_{(1)}f(v) S(h_{(2)})$. The category $_{H}^{H} \mathcal{YD}$ is braided
with braiding $c_{V,W}: V\otimes W \rightarrow W \otimes V$ given by $%
c_{V,W}(v \otimes w) = v_{\langle -1 \rangle}w \otimes v_{\langle 0\rangle}$.

For $C$ a coalgebra in $_{H}^{H}\mathcal{YD}$, we use a modified version of
the Heyneman-Sweedler notation, writing superscripts instead of subscripts,
so that comultiplication is written
\begin{equation*}
\Delta _{C}(x)=\Delta (x)\text{ }=x^{(1)}\otimes x^{(2)}\text{, for every }%
x\in C\text{.}
\end{equation*}

If $C$ and $D$ are coalgebras in $_{H}^{H} \mathcal{YD}$, so is $C
\underline{\otimes} D$ defined as follows. As a Yetter-Drinfeld module, $C
\underline{\otimes} D = C \otimes D$ with $H$-action and coaction as
described above. The counit is $\varepsilon_{C \underline{\otimes} D} =
\varepsilon_C \otimes \varepsilon_D$ and the comultiplication is $\Delta_{C
\underline{\otimes} D} = (C \otimes c_{C,D} \otimes D) \circ (\Delta_C
\otimes \Delta_D)$ , so that
\begin{eqnarray*}  \label{form: DeltaRR}
\Delta_{C \underline{\otimes} D}(x \otimes y) &=& (x^{(1)} \otimes
x^{(2)}_{\langle -1 \rangle }y^{(1)}) \otimes (x^{(2)}_{\langle0\rangle}
\otimes y^{(2)}). \\
\Delta _{C\underline{\otimes} D\underline{\otimes} E}\left( x\otimes
y\otimes z\right) &=& (x^{\left( 1\right) }\otimes x_{\left\langle
-2\right\rangle }^{\left( 2\right) }y^{\left( 1\right) }\otimes
x_{\left\langle -1\right\rangle }^{\left( 2\right) }y_{\left\langle
-1\right\rangle }^{\left( 2\right) }z^{\left( 1\right) }) \otimes
(x_{\left\langle 0\right\rangle }^{\left( 2\right) }\otimes y_{\left\langle
0\right\rangle }^{\left( 2\right) }\otimes z^{\left( 2\right) }).
\end{eqnarray*}
When it is clear from the context (and from the superscript versus subscript
notation) $\underline{\otimes}$ is written simply as $\otimes$. \vspace{2mm}

For a $K$-coalgebra $C$ and a map $u_C: K \rightarrow C$, the coalgebra $%
(C,u_C)$ is called \textbf{coaugmented} if $1_C := u_C(1_K)$ is a grouplike
element. For $C$ a coalgebra in $_{H}^{H} \mathcal{YD}$, then $u_C$ is also
required to be a map in the Yetter-Drinfeld category, i.e., for all $h \in H$%
,
\begin{equation}
{h}\cdot 1_C =\varepsilon _{H}(h)1_C \qquad \text{and}\qquad \rho _{C}(1_C)
=1_{H}\otimes 1_C.  \label{eq:YD0'}
\end{equation}
A coaugmented coalgebra $C$ is called connected if $C_0 = K1_C$.

\vspace{1mm}

The next definitions and lemmas will be key in later computations.

\begin{definition}
\label{def: Psi} For $M$ a left $H$-comodule, define $\Psi: \mathrm{Hom}%
(M,K) \rightarrow \mathrm{Hom}^{H,-}(M,H)$, the left $H$-comodule maps from $%
M$ to $H$, by
\begin{equation*}
\Psi \left( \alpha \right) =\left( H\otimes \alpha \right) \rho _{M}.
\end{equation*}
\end{definition}

\begin{remark}
\label{rem: hope1.8} Let $f: M \rightarrow L$ be a morphism of left $H$%
-comodules and $\alpha \in \mathrm{Hom}(L,K)$. Then $\Psi \left( \alpha
\right) \circ f =\Psi \left( \alpha \circ f \right).$ To see this, let $m
\in M$ and then
\begin{equation*}
(\Psi(\alpha)\circ f) (m) = \Psi(\alpha)(f(m)) = f(m)_{\left\langle
-1\right\rangle }\alpha(f(m)_{\left\langle 0\right\rangle }) =
m_{\left\langle -1\right\rangle }\alpha(f(m_{\left\langle 0\right\rangle }))
= \Psi(\alpha \circ f)(m).
\end{equation*}
\end{remark}

\begin{lemma}
\label{lem: Psi} For $C$ a left $H$-comodule coalgebra, $\Psi :\mathrm{Hom}%
\left( C,K\right) \rightarrow \mathrm{Hom}^{H,-}\left( C,H\right)$ is an
algebra isomorphism.
\end{lemma}

\begin{proof}
Let ${^{H}\mathfrak{M}}$ denote the monoidal category of left $H$-comodules.
Then the functor $\left( -\right) \otimes H:Vec\left( K\right) \rightarrow {%
^{H}\mathfrak{M}}$ is right adjoint to the forgetful functor. The canonical
isomorphism defining the adjunction yields $\Psi $. Explicitly, let $\alpha
,\beta \in \mathrm{Hom}\left( C,K\right) $ and let $z\in C$. We have%
\begin{eqnarray*}
\left[ \Psi \left( \alpha \right) \ast \Psi \left( \beta \right) \right]
\left( z\right) &=&\left[ \left( H\otimes \alpha \right) \rho _{C}\right]
\ast \left[ \left( H\otimes \beta \right) \rho _{C}\right] \left( z\right) \\
&=&\left( H\otimes \alpha \right) \rho _{C}\left( z^{\left( 1\right)
}\right) \left[ \left( H\otimes \beta \right) \rho _{C}\right] \left(
z^{(2)}\right) \\
&=&z_{\left\langle -1\right\rangle }^{\left( 1\right) }\alpha \left(
z_{\left\langle 0\right\rangle }^{\left( 1\right) }\right) z_{\left\langle
-1\right\rangle }^{\left( 2\right) }\beta \left( z_{\left\langle
0\right\rangle }^{\left( 2\right) }\right) \\
&=&z_{\left\langle -1\right\rangle }^{\left( 1\right) }z_{\left\langle
-1\right\rangle }^{\left( 2\right) }\alpha \left( z_{\left\langle
0\right\rangle }^{\left( 1\right) }\right) \beta \left( z_{\left\langle
0\right\rangle }^{\left( 2\right) }\right) \\
&=&z_{\left\langle -1\right\rangle }\alpha \left[ \left( z_{\left\langle
0\right\rangle }\right) ^{\left( 1\right) }\right] \beta \left[ \left(
z_{\left\langle 0\right\rangle }\right) ^{\left( 2\right) }\right] \\
&=&z_{\left\langle -1\right\rangle }\left( \alpha \ast \beta \right) \left(
z_{\left\langle 0\right\rangle }\right) =\Psi \left( \alpha \ast \beta
\right) \left( z\right)
\end{eqnarray*}%
and also $\Psi \left( \varepsilon _{C}\right) \left( z\right) =\left(
H\otimes \varepsilon _{C}\right) \rho _{C}=u_{H}\varepsilon _{C}$. Thus $%
\Psi $ is an algebra homomorphism.

For $\alpha \in \mathrm{Hom}\left( C,K\right) $, then $\Psi \left( \alpha
\right)$ is a morphism in ${^{H}\mathfrak{M}}$ since
\begin{equation*}
\Delta _{H}\Psi \left( \alpha \right) \left( z\right) =\sum \Delta
_{H}\left( z_{\left\langle -1\right\rangle }\right) \alpha \left(
z_{\left\langle 0\right\rangle }\right) =\sum z_{\left\langle
-2\right\rangle }\otimes z_{\left\langle -1\right\rangle }\alpha \left(
z_{\left\langle 0\right\rangle }\right) =\sum z_{\left\langle
-1\right\rangle }\otimes \Psi \left( z_{\left\langle 0\right\rangle }\right)
.
\end{equation*}%
Finally, the composition inverse of $\Psi $ is $\Psi ^{-1}$ where $\Psi
^{-1} \left( \sigma \right) :=\varepsilon _{H}\sigma .$
\end{proof}

\vspace{2mm}

\begin{remark}
\label{rem: v e v^-1 lin}

(i) Let $(C,\Delta _{C},\varepsilon _{C},u_{C})$ be a coaugmented connected
coalgebra, $A$ an algebra and $f:C\rightarrow A$ a map such that $f(1_C) =
1_A$, i.e., $f = u_A \varepsilon_C$ on the coradical of $C$. Then by \cite[%
Lemma 5.2.10]{Mo}, $f$ is convolution invertible with inverse $%
\sum_{n=0}^{\infty}\gamma^n$ where $\gamma = u_A \varepsilon_C - f$. Then
since $\gamma^{n+1} = 0$ on $C_n$, it is also true that $f^{-1}=u_A
\varepsilon_C $ on $C_0$.

(ii) For $C$ a left $H$-module coalgebra and $\upsilon: C \rightarrow K$
convolution invertible, then $\upsilon$ is left $H$-linear if and only if $%
\upsilon^{-1}$ is. For suppose $\upsilon$ is left $H$-linear. Then
\begin{equation*}
\upsilon^{-1}(hz)=\upsilon^{-1}(hz^{\left( 1\right) })\upsilon( z^{\left(
2\right) }) \upsilon^{-1}( z^{\left( 3\right) }) =\upsilon^{-1}(h_{\left(
1\right) }z^{\left( 1\right) })\upsilon ( h_{\left( 2\right) }z^{\left(
2\right) }) \upsilon^{-1}( z^{\left( 3\right) }) =\varepsilon _{H}\left(
h\right)\upsilon^{-1}\left( z\right),
\end{equation*}
and so $\upsilon^{-1}$ is left $H$-linear also. \qed
\end{remark}

The next condition is part of the definition of a cocycle $\xi$ for a
pre-bialgebra {\ \cite{A.M.St.-Small} (see Section \ref{sec: pre-bialg} )
but makes sense for any coalgebra $C$ in $_{H}^{H}\mathcal{YD}$. }

\begin{definition}
\label{de: dualSweedler 1 cocycle} Let $C$ be a coalgebra in $_{H}^{H}%
\mathcal{YD}$ and $\alpha \in \mathrm{Hom}(C,H)$. Then we say that $\alpha$
is a dual normalized Sweedler 1-cocycle if $\Delta _{H}\alpha =(m_{H}\otimes
\alpha )(\alpha \otimes \rho _{C})\Delta _{C}$ and $\varepsilon _{H}\alpha =
\varepsilon_C$. Thus for $x \in C$,
\begin{equation}  \label{eq: Sweedler 1-cocycle}
\alpha(x)_{(1)} \otimes \alpha(x)_{(2)} = \alpha(x^{(1)}){x^{(2)}}_{\langle
-1 \rangle} \otimes \alpha({x^{(2)}}_{\langle 0 \rangle}) \quad \text{and}%
\quad \varepsilon_H(\alpha(x)) = \varepsilon_C(x).
\end{equation}
\end{definition}

\vspace{1mm} If $\alpha: C \rightarrow H$ is a dual normalized Sweedler
1-cocycle, then $\alpha$ is convolution invertible and its inverse can be
described explicitly.

\begin{proposition}
\label{pro: inverse xi}Let $C$ be a coalgebra in $_{H}^{H}\mathcal{YD}$ and
let $\alpha: C \rightarrow H$ satisfy (\ref{eq: Sweedler 1-cocycle}). Then $%
\alpha^\prime$ is the convolution inverse of $\alpha$ where
\begin{equation*}
\alpha ^{\prime}:=m_{H}\circ \left( H\otimes S_H \circ \alpha \right) \circ
\rho _{C}.
\end{equation*}
\end{proposition}

\begin{proof}
For any $x \in C $, we have
\begin{eqnarray*}
\alpha \ast \alpha ^{\prime }(x) &=& \alpha(x^{(1)})\alpha^\prime(x^{(2)}) =
\alpha(x^{(1)})x^{(2)}_{\langle -1 \rangle} S_H(\alpha(x^{(2)}_{\langle
0\rangle})) \overset{\text{(\ref{eq: Sweedler 1-cocycle})}}{=}
\alpha(x)_{(1)}S_H(\alpha(x)_{(2)}) \\
&=& \varepsilon_H \circ \alpha (x) \overset{\text{(\ref{eq: Sweedler
1-cocycle})}}{=} u_{H}\circ \varepsilon _{C}(x), \hspace{2mm} \text{ and }
\end{eqnarray*}
\begin{eqnarray*}
\alpha^{\prime } * \alpha(x) &=& \alpha^{\prime }(x^{(1)})\alpha(x^{(2)}) =
x^{(1)}_{\langle-1\rangle}S_H(\alpha(x^{(1)}_{\langle 0
\rangle}))\alpha(x^{(2)}) \\
&=& x^{(1)}_{\langle-1\rangle} x^{(2)}_{\langle-2\rangle} S_H(
x^{(2)}_{\langle-1\rangle})S_H( \alpha ((x^{(1)})_{\langle 0 \rangle}
))\alpha((x^{(2)})_{\langle 0 \rangle}) \\
&=& x_{\langle-1\rangle} S_H(( x^{(2)}_{\langle
0\rangle})_{\langle-1\rangle})S_H(\alpha(x_{\langle 0
\rangle}^{(1)}))\alpha((x^{(2)}_{\langle 0\rangle})_{\langle 0\rangle}) \\
&=& x_{\langle-1\rangle} S_H[\alpha(x_{\langle 0 \rangle}^{(1)})(
x^{(2)}_{\langle 0\rangle})_{\langle-1\rangle}] \alpha ((x^{(2)}_{\langle
0\rangle})_{\langle 0\rangle}) \\
&\overset{\text{(\ref{eq: Sweedler 1-cocycle})}}{=}& x_{\langle-1\rangle}
S_H(\alpha(x_{(0)})_{(1)})\alpha(x_{(0)})_{(2)}\\
&=& x_{\langle-1\rangle} \varepsilon_H(\alpha(x_{(0)}) =
x_{\langle-1\rangle} \varepsilon_{C}(x_{(0)}) \\
&=& u_H \circ \varepsilon_{C}(x).
\end{eqnarray*}
Thus $\alpha^\prime $ is the convolution inverse of $\alpha$ as claimed.
\end{proof}

The next definition/lemma will be useful in upcoming computations.

\begin{lemma}
\label{lem: Phi} Let $C $ be a coalgebra and let $(M, \mu)$ be a left $H$%
-module. Define
\begin{equation*}
\Phi :\mathrm{Hom} \left( C,H\right) \rightarrow \mathrm{End} \left(
C\otimes M\right) \text{ by } \Phi \left( \alpha \right) :=\left( C\otimes
\mu _{M}\right) \left[ \left( C\otimes \alpha \right) \Delta _{C}\otimes M%
\right],
\end{equation*}
for $\alpha \in \mathrm{Hom}(C,H)$. The map $\Phi $ is an algebra
homomorphism.
\end{lemma}

\begin{proof}
By definition,
\begin{equation*}
\Phi(\alpha)(x \otimes m) = x_{(1)} \otimes \alpha(x_{(2)})m \text{ for } x
\in C, m \in M.
\end{equation*}
Then for $\alpha ,\beta \in \mathrm{Hom} \left( C ,H\right)$,
\begin{eqnarray*}
(\Phi(\alpha)\circ\Phi(\beta))(x \otimes m) &=& \Phi(\alpha)\ (x_{(1)}
\otimes \beta(x_{(2)})m) = x_{(1)} \otimes \alpha(x_{(2)})\beta(x_{(3)})m \\
&=& x_{(1)} \otimes (\alpha * \beta)(x_{(2)})m = \Phi(\alpha * \beta)(x
\otimes m),
\end{eqnarray*}
and
\begin{equation}
\Phi(u_H \circ \varepsilon_C)(x \otimes m ) = x_{(1)} \otimes (u_H \circ
\varepsilon_C)(x_{(2)}) m = x \otimes m.  \notag
\end{equation}
\end{proof}

\begin{remark}
\label{re: Phi} Note that $\Phi(\alpha) = \text{Id}_{C \otimes M}$ if and
only if the action of $\alpha(C)$ on $M$ is trivial. For if $x \otimes m
=x_{(1)} \otimes \alpha(x_{(2)})m$, apply $\varepsilon_C$ to the left hand
tensorand to see that $\alpha(C)$ acts trivially on $M$. The converse is
obvious. \qed
\end{remark}

\subsection{Pre-bialgebras with cocycle}

In this section, we recall the notion of a pre-bialgebra with cocycle in $%
_{H}^{H}\mathcal{YD}$ and explain how a pre-bialgebra with cocycle is
associated to a splitting datum. Throughout this section, $H$ will denote a
Hopf algebra over $K$.

\begin{definition}
(cf. \cite[Definition 1.8]{AM-Small2}) A \textbf{splitting datum} $\left(
A,H,\pi ,\sigma \right) $ consists of a bialgebra $A,$ a bialgebra
homomorphism $\sigma :H\rightarrow A$ and an $H$-bilinear coalgebra
homomorphism $\pi :A\rightarrow H$ such that $\pi \sigma =\mathrm{Id}_{H}$.
Note that $H$-bilinear here means $\pi(\sigma(h)x \sigma(h^{\prime})) =
h\pi(x)h^{\prime}$ for all $h,h^{\prime}\in H$ and $x\in A$. We say that a
splitting datum is \textbf{trivial} whenever $\pi $ is a bialgebra
homomorphism.
\end{definition}

\begin{example}
Let $H$ be a Hopf algebra and let $A=R\#H$ be the usual Radford biproduct of
a bialgebra $R$ in $^H_H\mathcal{YD}$ and $H$. Let $\sigma:H\rightarrow
A,\sigma(h)=1_R\# h,$ and let $\pi:A\rightarrow
H,\pi(r\#h)=r\varepsilon_H(h).$ Then $(A,H,\pi,\sigma)$ is a trivial
splitting datum. Conversely, for $A$ a Hopf algebra, if $(A,H,\pi,\sigma)$
is a trivial splitting datum, then $A$ is isomorphic to a Radford biproduct
or bosonization of $H$ (identified with $\sigma(H)$) and the $K$-algebra $R$
of $\pi$-coinvariants, a Hopf algebra in the category ${\ _{H}^{H}\mathcal{YD%
}}$.
\end{example}

If $\pi$ is not an algebra map, then $R = \{ x \in A| (A \otimes
\pi)\Delta(x) = x \otimes 1 \}$ need not be a Hopf algebra in ${_{H}^{H}%
\mathcal{YD}}$ but instead will be a \textit{pre-bialgebra with a cocycle in
$^H_H\mathcal{YD}$}.

\subsubsection{Definition of a pre-bialgebra and a pre-bialgebra with cocycle%
}

\label{sec: pre-bialg} Following \cite[Definition 2.3, Definitions 3.1]%
{A.M.St.-Small}, we define the following. A pre-bialgebra $R$ = $%
(R,m_R,u_R,\Delta_R ,\varepsilon_R )$ in ${_{H}^{H}\mathcal{YD}}$ is a
coaugmented coalgebra $\left( R,\Delta_R ,\varepsilon_R, u_R \right) $ in
the category ${_{H}^{H}\mathcal{YD}} $ together with a left $H$-linear map $%
m_R:R\otimes R\rightarrow R $ such that $m_{R}$ is a coalgebra homomorphism,
i.e,
\begin{equation}
\Delta_R m_{R}=(m_{R}\otimes m_{R})\Delta _{R\otimes R}\qquad \text{and}
\qquad \varepsilon_R m_{R}=m_{K}(\varepsilon_R \otimes \varepsilon_R ),
\end{equation}

and $u_R$ is a unit for $m_{R}$, i.e.,
\begin{equation}
m_{R}(R\otimes u_R)=R=m_{R}(u_R \otimes R).  \label{eq:YD9'}
\end{equation}

When clear from the context, the subscript $R$ on the structure maps above
is omitted. \vspace{1mm}

Essentially a pre-bialgebra differs from a bialgebra in $_H^H\mathcal{YD}$
in that \textbf{the multiplication need not be associative and need not be a
morphism of $H$-comodules}, see Example \ref{ex: pbw}. \vspace{2mm}

A pre-bialgebra \emph{with cocycle }in ${_{H}^{H}\mathcal{YD}}$ is a pair $%
(R,\xi )$ where $R=(R,m,u,\Delta ,\varepsilon )$ is a pre-bialgebra in ${%
_{H}^{H}\mathcal{YD}}$ and $\xi: C=R \otimes R \rightarrow H $ is a
normalized dual Sweedler 1-cocycle (\ref{eq: Sweedler 1-cocycle}), left $H$%
-linear with respect to the left adjoint action of $H$ on $H$, and, for all $%
r,s\in R$ and $h\in H$, the following hold:
\begin{eqnarray}
&&c_{R,H}(m\otimes \xi )\Delta _{R\otimes R}=(m_{H}\otimes m_{R})(\xi
\otimes \rho _{R\otimes R})\Delta _{R\otimes R};  \label{eq:YD6'} \\
&& m_{R}(R\otimes m_{R})= m_{R}(R\otimes \mu _{R})[(m_{R}\otimes \xi )\Delta
_{R\otimes R}\otimes R] = m_R(m_R \otimes R)\Phi(\xi);  \label{eq:YD7'} \\
&&m_{H}(\xi \otimes H)[R\otimes (m_{R}\otimes \xi )\Delta _{R\otimes
R}]=m_{H}(\xi \otimes H)(R\otimes c_{H,R})[(m_{R}\otimes \xi )\Delta
_{R\otimes R}\otimes R];  \label{eq:YD8'} \\
&&\xi (R\otimes u)=\xi (u\otimes R)=\varepsilon 1_{H}.  \label{eq:YD10'}
\end{eqnarray}

By definition \cite{AM-Small2}, a map $f:(\left( R,m,u,\delta ,\varepsilon
\right) ,\xi )\rightarrow \left( (R^{\prime },m^{\prime },u^{\prime },\delta
^{\prime },\varepsilon ^{\prime }),\xi ^{\prime }\right) $\ of
pre-bialgebras with a cocycle in ${_{H}^{H}\mathcal{YD}}$, is a morphism of
pre-bialgebras with cocycle if $f$ is a coalgebra homomorphism $f:\left(
R,\delta ,\varepsilon \right) \rightarrow \left( R^{\prime },\delta ^{\prime
},\varepsilon ^{\prime }\right) $ in the category $({_{H}^{H}\mathcal{YD}}%
,\otimes ,K)$ such that%
\begin{equation*}
f\circ m=m^{\prime }\circ \left( f\otimes f\right) ,\qquad f\circ
u=u^{\prime }\qquad \text{and}\qquad \xi ^{\prime }\circ \left( f\otimes
f\right) =\xi .
\end{equation*}%
\vspace{1.5mm}

\begin{remark}
\label{rem: associativity}In (\ref{eq:YD7'}) the map $\Phi$ from Lemma \ref%
{lem: Phi} is used with $C=R \otimes R$ and $M=R$. Thus if $\Phi(\xi)$ is
the identity, or, equivalently, if $\xi(R \otimes R)\subseteq H$ acts
trivially on $R$, then $m_R$ is associative. We will see in Section \ref%
{sec: xi} that if $R$ is connected, the converse holds. \qed
\end{remark}

\vspace{1.5mm} Let $(R,\xi)$ be a pre-bialgebra with cocycle in ${_{H}^{H}%
\mathcal{YD}}$. Since $\xi: C = R \otimes R \rightarrow H $ is a dual
Sweedler 1-cocycle, by Proposition \ref{pro: inverse xi}, $\xi$ is
convolution invertible with convolution inverse
\begin{equation}  \label{form: xi
inverse}
\xi^{-1} = m_H \circ (H \otimes S_H \circ \xi) \circ \rho_{R \otimes R}.
\end{equation}

Thus equation (\ref{eq:YD7'}) is equivalent to:
\begin{equation}
m_{R}\circ \left( m_{R}\otimes R\right) = m_{R}\circ (R\otimes m_{R})\circ
\Phi \left( \xi ^{-1}\right).  \notag
\end{equation}

\vspace{1.5mm}

If, as well, $R$, and thus $R \otimes R$ is connected, since $\xi(1_{R
\otimes R}) = u_H(1_K)$, by Remark \ref{rem: v e v^-1 lin}(i) with $C = R
\otimes R$ and $A=H$, then another form for $\xi^{-1}$ is
\begin{equation*}
\xi^{-1} = \sum_{n=0}^{\infty} ( u_H \circ \varepsilon_C - \xi)^n.
\end{equation*}

\vspace{2mm}

\subsubsection{The splitting datum associated to a pre-bialgebra with cocycle%
}

\label{sec: splittingdatum} To every $(R, \xi)$, we can associate a
splitting datum $(A:=R \#_\xi H ,H,\pi,\sigma)$ where the bialgebra $R
\#_\xi H$ is constructed as follows (see \cite[Theorem 3.62]{A.M.S.} and
\cite[Definitions 3.1]{A.M.St.-Small} ). As a vector space, $A=R\otimes H$
with coalgebra and algebra structures given below.

Let $r,s \in R$, $h,h^\prime \in H$. The coalgebra structures are $%
\varepsilon _{A}\left( r\#h\right) = \varepsilon_R \left( r\right)
\varepsilon _{H}\left( h\right)$, and
\begin{equation}
\Delta _{A}\left( r\#h\right) = r^{(1)}\#r_{\langle -1\rangle
}^{(2)}h_{(1)}\otimes r_{\langle 0\rangle }^{(2)}\#h_{\left( 2\right)
},\qquad \text{where }\Delta_R (r)= r^{(1)}\otimes r^{(2)}.
\label{eq: DeltaA}
\end{equation}
In other words, as a coalgebra, $A$ is the smash coproduct of $R$ and $H$.

For future calculations, we note:
\begin{eqnarray}  \label{form: DeltaAr}
\Delta _{A}\left( r\#1_{H}\right) &=& r^{(1)}\#r_{\langle -1\rangle
}^{(2)}\otimes r_{\langle 0\rangle }^{(2)}\#1_{H} \\
\Delta _{A}^{2}\left( r\#1_{H}\right) &=& r^{(1)}\#r_{\langle -1\rangle
}^{(2)}r_{\langle -2\rangle }^{(3)}\otimes r_{\langle 0\rangle
}^{(2)}\#r_{\langle -1\rangle }^{(3)}\otimes r_{\langle 0\rangle
}^{(3)}\#1_{H}.  \label{form: Delta2Ar}
\end{eqnarray}

The unit is $u_{A}(1) = 1_R\#1_{H}$ and the multiplication is given by
\begin{equation*}
m_{A} = \left( R\otimes m_{H}\right) \left[ \left( m_{R}\otimes \xi \right)
\Delta _{R\otimes R}\otimes m_{H}\right] \left( R\otimes c_{H,R}\otimes
H\right)
\end{equation*}
so that for $r,s \in R$, $h, h^\prime \in H$,
\begin{equation}
m_{A}(r \#h \otimes s \#h^{\prime }) = m_R(r^{(1)} \otimes r_{\langle -1
\rangle}^{(2)}(h_{(1)}s)^{(1)}) \# \xi(r^{(2)}_{\langle 0 \rangle} \otimes
(h_{(1)}s)^{(2)})h_{(2)}h^{\prime }.  \label{form: multi
smash xi}
\end{equation}
Using the map $\Phi$ from Lemma \ref{lem: Phi} with $C = R \otimes R$ and $M
= (H,m_H)$, we write:
\begin{eqnarray}
m_{R \#_\xi H} &= & ( m_R \otimes m_{H})\circ (\Phi(\xi)\otimes H) \circ
\left( R\otimes c_{H,R}\otimes H\right)  \notag \\
& =& (m_R \otimes H )\circ \Phi(\xi) \circ (R\otimes R \otimes m_H) \circ
\left( R\otimes c_{H,R}\otimes H\right) .  \label{eq:
BmultPhi}
\end{eqnarray}
Here, unless $\xi(R \otimes R) = K$, the action of $\xi(R \otimes R)$ will
not be trivial and $\Phi(\xi)$ will not be the identity. It will be useful
to have the following formulas:
\begin{eqnarray}  \label{form: varepsilonHmA}
(R \otimes \varepsilon_H)m_A(r \# h \otimes s \# h^{\prime }) &=& m_R(r
\otimes hs)\varepsilon(h^{\prime }); \\
(\varepsilon_R \otimes H)m_A(r \# h \otimes s \# h^{\prime }) &=& \xi(r
\otimes h_{(1)}s)h_{(2)}h^{\prime }.  \label{form:
varepsilonRmA}
\end{eqnarray}

Note that the canonical injection $\sigma :H\hookrightarrow R\#_{\xi }H$ is
a bialgebra homomorphism. \label{lem: sigma bialg}Furthermore
\begin{equation*}
\pi :R\#_{\xi }H\rightarrow H:r\#h\longmapsto \varepsilon \left( r\right) h
\end{equation*}
is an $H$-bilinear coalgebra retraction of $\sigma $.

\subsubsection{The pre-bialgebra with cocycle associated to a splitting datum%
}

\label{sec: A co H}

Suppose that $(A,H,\pi ,\sigma )$ is a splitting datum. In this subsection
we describe $(R,\xi )$ , the associated pre-bialgebra with cocycle in $%
_{H}^{H}\mathcal{YD}$ \cite[Definitions 3.2]{A.M.St.-Small}. As when $\pi $
is a bialgebra morphism and $A$ is a Radford biproduct, set
\begin{equation*}
R=A^{co\pi }=\left\{ a\in A\mid a_{\left( 1\right) }\otimes \pi \left(
a_{\left( 2\right) }\right) =a\otimes 1_{H}\right\} ,
\end{equation*}%
and let
\begin{equation*}
\tau :A\rightarrow R,\hspace{1mm}\tau \left( a\right) =a_{\left( 1\right)
}\sigma S\pi \left( a_{\left( 2\right) }\right) .
\end{equation*}%
Define a left-left Yetter-Drinfeld structure on $R$ by
\begin{equation*}
h\cdot r=hr=\sigma \left( h_{\left( 1\right) }\right) r\sigma S_{H}\left(
h_{\left( 1\right) }\right) ,\text{\qquad }\rho \left( r\right) =\pi \left(
r_{\left( 1\right) }\right) \otimes r_{\left( 2\right) },
\end{equation*}%
and define a coalgebra structure in $_{H}^{H}\mathcal{YD}$ on $R$ by
\begin{equation}
\Delta _{R}(r)=r^{(1)}\otimes r^{(2)}=r_{(1)}\sigma S\pi (r_{(2)})\otimes
r_{(3)}=\tau \left( r_{\left( 1\right) }\right) \otimes r_{\left( 2\right) },%
\text{\qquad }\varepsilon =\varepsilon _{A\mid R}.  \label{form: comulti R}
\end{equation}

The map
\begin{equation*}
\omega :R\otimes H\rightarrow A\text{, }\omega (r\otimes h)=r\sigma (h)
\end{equation*}%
is an isomorphism of $K$-vector spaces, the inverse being defined by%
\begin{equation*}
\omega ^{-1}:A\rightarrow R\otimes H\text{, }\omega ^{-1}(a)=a_{\left(
1\right) }\sigma S_{H}\pi \left( a_{\left( 2\right) }\right) \otimes \pi
\left( a_{\left( 3\right) }\right) =\tau \left( a_{\left( 1\right) }\right)
\otimes \pi \left( a_{\left( 2\right) }\right) .
\end{equation*}%
Clearly $A$ defines, via $\omega $, a bialgebra structure on $R\otimes H$
that will depend on $\sigma $ and $\pi $. As shown in \cite[6.1]{Scha} and
\cite[Theorem 3.64]{A.M.S.}, $(R,m,u,\Delta ,\varepsilon )$ is a
pre-bialgebra in ${_{H}^{H}\mathcal{YD}}$ with cocycle $\xi $ where the maps
$u:K\rightarrow R$ and $m:R\otimes R\rightarrow R$, are defined by%
\begin{equation*}
u=u_{A}^{\mid R},\text{\qquad }m(r\otimes s)=r_{(1)}s_{(1)}\sigma S\pi
(r_{(2)}s_{(2)})=\tau \left( r\cdot _{A}s\right)
\end{equation*}%
and the cocycle $\xi :R\otimes R\rightarrow H$ is the map defined by
\begin{equation*}
\xi (r\otimes s)=\pi (r\cdot _{A}s).
\end{equation*}%
Then $(R,\xi)$ is the pre-bialgebra with cocycle in $_{H}^{H} \mathcal{YD}$
associated to $\left( A,H,\pi ,\sigma \right) $.

We note that the map $\tau $ above is a surjective coalgebra homomorphism
and satisfies the following where $a\in A,h\in H$ and $r,s\in R.$ \cite[%
Proposition 3.4]{A.M.St.-Small}:
\begin{eqnarray*}
\tau \left[ a\sigma \left( h\right) \right] &=&\tau \left( a\right)
\varepsilon _{H}\left( h\right) ,\qquad \tau \left[ \sigma \left( h\right) a%
\right] =\text{ }h\cdot \tau \left( a\right) , \\
r\cdot _{R}s &=&\tau \left( r\cdot _{A}s\right) ,\qquad \tau \left( a\right)
\cdot _{R}\tau \left( b\right) =\tau \left[ \tau \left( a\right) \cdot _{A}b%
\right] .
\end{eqnarray*}
Note that $h\cdot r=\tau \left( \sigma \left( h\right) r\right) $ for all $%
h\in H,r\in R.$

\subsubsection{The correspondence between splitting data and pre-bialgebras
with cocycle}

\label{sec: correspondence}

If we start with a splitting datum $(A,H,\pi,\sigma)$, and construct $%
(R,\xi) $ as in Section \ref{sec: A co H}, and then construct the splitting
datum $(R \#_\xi H, H, \pi,\sigma)$ associated to $(R,\xi)$ as in Section %
\ref{sec: splittingdatum}, then (cf. \cite[6.1]{Scha}) $\omega :R\#_{\xi
}H\rightarrow A$ is a bialgebra isomorphism.

Conversely, we start with a pre-bialgebra with cocycle $(R,\xi )$ in ${%
_{H}^{H}\mathcal{YD}}$ and construct the splitting datum $(R\#_{\xi }H,H,\pi
,\sigma )$
\begin{equation*}
\sigma :H\hookrightarrow R\#_{\xi }H\text{\qquad and\qquad }\pi :R\#_{\xi
}H\rightarrow H
\end{equation*}%
as in Section \ref{sec: splittingdatum}. Then $(R\#_{\xi }H)^{co\pi} =R \# K
$ and so the pre-bialgebra in ${_{H}^{H}\mathcal{YD}}$ associated to $\left(
R\#_{\xi }H,H,\pi ,\sigma \right) $ constructed in Section \ref{sec: A co H}
is $R \otimes K$ which is isomorphic to $R$ as a coalgebra in $_{H}^{H}%
\mathcal{YD}$ via the map $\theta:R\otimes K\rightarrow R $ where $\theta
(r\otimes 1)=r$. The corresponding cocycle is $\xi ^{\prime }$ where {$\xi
^{\prime }=$}$\xi \left(\theta\otimes \theta\right) $. Clearly $\theta$
induces an isomorphism of pre-bialgebras with cocycle between $\left(
R\otimes K,{\xi ^{\prime }}\right) $ and $\left( R,{\xi }\right).$

In this situation we note that $\tau:(R\#_{\xi }H) \rightarrow (R\#_{\xi
}H)^{co\pi}$ is given by $R \otimes \varepsilon$. For we have that
\begin{eqnarray*}
\tau(r \#h) &\overset{(\ref{eq: DeltaA})}{=}& (r^{(1)} \# r^{(2)}_{\langle
-1 \rangle}h_{(1)}) \sigma S \pi(r^{(2)}_{\langle 0 \rangle}\# h_{(2)}) \\
&=& (r^{(1)} \# r^{(2)}_{\langle -1 \rangle}h_{(1)}) \sigma S(h_{(2)})
\varepsilon(r^{(2)}_{\langle 0 \rangle}) \\
&=& r \# \varepsilon(h).
\end{eqnarray*}

\section{Associativity of $(R,\protect\xi)$}

\label{sec: xi} In general, the multiplication in a pre-bialgebra $R$
associated to a splitting datum $(A,H,\pi, \sigma)$ need not be associative.
It was noted in the previous section that $m_R$ is associative if $\Psi(\xi)$
is the identity, or, equivalently, if $\xi(R \otimes R) \subset H$ acts
trivially on $R$. First we consider some examples, and then show that the
converse statement holds if $R$ is connected.

\begin{example}
\textit{Thin pre-bialgebras} A pre-bialgebra $R$ is called \textit{thin} if $%
R$ is connected and the space of primitives of $R$ is also one-dimensional.
By \cite[Theorem 3.14]{A.M.St.-Small}, a finite dimensional thin
pre-bialgebra $(R,\xi)$ has associative multiplication, even for nontrivial $%
\xi$, but need not be a bialgebra in $_H^H\mathcal{YD}$ by \cite[Example 6.4]%
{AM-Small2}.
\end{example}

For $\Gamma$ a finite abelian group, $V \in {_{\Gamma }^{\Gamma }\mathcal{YD}%
}$ (where we write ${_{\Gamma }^{\Gamma }\mathcal{YD}}$ for ${_{H }^{H}%
\mathcal{YD}}$ with $H = K[\Gamma]$), and $\mathcal{B}(V)$ the Nichols
algebra of $V$, then all Hopf algebras whose associated graded Hopf algebra
is $\mathcal{B}(V) \# K[\Gamma]$, called the liftings of $\mathcal{B}(V) \#
K[\Gamma]$, are well-known if $V$ is a quantum linear space. Recall that for
$g \in G$ and $\chi \in \widehat{\Gamma}$, $V_{g }^{\chi }$ is the set of $v
\in V$ such that $\rho(v) = v_{\langle -1 \rangle} \otimes v_{\langle 0
\rangle} = g \otimes v$ and $h \cdot v = \chi(h)v$ for all $h \in G$.

\begin{definition}
\label{de: qls} $V=\oplus _{i=1}^{t}Kv_{i}\in {_{\Gamma }^{\Gamma }\mathcal{%
YD}}$ with $0\neq v_{i}\in V_{g_{i}}^{\chi _{i}}$ with $g_i \in \Gamma$, $%
\chi_i \in \widehat{\Gamma}$ is called a quantum linear space when
\begin{equation*}
\chi _{i}(g_{j})\chi _{j}(g_{i})=1\text{ for }i\neq j\mbox{ and
}\chi _{i}(g_{i})\text{ is a primitive }r_{i}th\mbox{ root of
}1,r_{i}>1.
\end{equation*}
\end{definition}

The following is proved in \cite{AS1} or \cite{bdgconstructing}.

\begin{proposition}
\label{pr: lift} For $V$ a quantum linear space with $\chi_i(g_i)$ a
primitive $r_i$th root of 1, all liftings $A:= A(a_i,a_{ij}| 1 \leq i,j\leq
t)$ of $\mathcal{B}(V)\#K[\Gamma]$ are Hopf algebras generated by the
grouplikes and by $(1,g_i)$-primitives $x_i$, $1 \leq i \leq t$ where
\begin{eqnarray*}
hx_i &=& \chi_i(h)x_ih; \\
x_i^{r_i} &=& a_i(1 - g_i^{r_i} ) \text{ where } a_i = 0 \text{ if }
g_i^{r_i}=1 \mbox{ or } \chi_i^{r_i} \neq \varepsilon; \\
x_ix_j &=& \chi_j(g_i)x_jx_i + a_{ij}(1 - g_ig_j ) \mbox{ where } a_{ij} = 0 %
\mbox{ if }g_ig_j=1 \mbox{ or } \chi_i\chi_j \neq \varepsilon.
\end{eqnarray*}
\end{proposition}

One sees directly from Proposition \ref{pr: lift} that $a_{ji} = -
\chi_j(g_i)^{-1}a_{ij} = - \chi_i(g_j)a_{ij}$. By rescaling, we
may assume that the $a_i$ are 0 or 1.\vspace{2mm}

\begin{example}
\label{ex: pbw} Using the notation of Proposition \ref{pr: lift}, let $%
A:=A(a_1,a_2,a_{12}=a)$ be a nontrivial lifting of $B(V) \# K[\Gamma]$ where
$V =Kx_1 \oplus Kx_2$ is a quantum plane as above. Then $A$ has PBW basis $%
\{gx_1^ix_2^j |g \in \Gamma, 0 \leq i,j \leq r-1 \}$, and the map $\pi: A
\rightarrow H= K[\Gamma]$ defined by $\pi(gx_1^ix_2^j ) = \delta_{0, i+j}g$
is an $H$-bilinear coalgebra homomorphism that splits the inclusion $H
\overset{\sigma}{\hookrightarrow }A $. Thus $(A,H,\pi,i)$ is a splitting
datum and so $A \cong R \#_\xi H$ for some pre-bialgebra with cocycle $%
(R,\xi)$. Since $R = A^{co\pi}$ then $R$ has $K$-basis $\{x_1^ix_2^j| 0 \leq
i,j \leq r-1 \}$. In general, $(R,\xi)$ is not associative.
\end{example}


The next example shows that for $A = A(1,1,a)$ as above, with $r >2$, there
is no choice of an $H$-bilinear projection $\pi$ splitting the inclusion
which will make the associated pre-bialgebra $(R,\xi)$ associative.

\begin{example}
\label{ex: qlp}

Let $A:= A(1,1,a)$ be the Hopf algebra described in Proposition \ref{pr:
lift} with $t=2$, $a_1 = a_2 =1$, $a_{12}=a \neq 0$ and $r:=r_1=r_2 >2$.
Then $\chi_1 = \chi_2^{-1}$, and $1 = \chi_1(g_2)\chi_2(g_1) $ implies that $%
\chi_1(g_1) = \chi_2(g_2)^{-1}$. Let $q$ denote $\chi_1(g_1)$. Then $x_2x_1
= q x_1x_2 + a(1-g_1g_2 )$. We show that there is no $H$-bilinear coalgebra
morphism $\pi: A \rightarrow K[\Gamma] = H$ splitting the inclusion $\sigma$
such that $R= A^{co\pi}$ is associative. The proof is by contradiction.

Suppose that $\pi$ is such a morphism and $R = A^{co\pi}$ is associative.
Then since $g_1$ is an invertible element, and
\begin{equation*}
g_1\pi(x_1^nx_2^m) = \pi(g_1 x_1^nx_2^m) = q^{n-m} \pi(x_1^nx_2^m g_1) =
q^{n-m} \pi(x_1^nx_2^m) g_1 = q^{n-m} g_1\pi(x_1^nx_2^m ),
\end{equation*}
then
\begin{equation}  \label{ndifferentfromm}
\pi(x_1^nx_2^m)=0 \text{ if } n \neq m, \text{ and } 0 \leq n,m \leq r-1.
\end{equation}

Next note that if $\pi :A\rightarrow H$ is as above, and if $\pi
(x_{1}^{i}x_{2}^{j})=0$ for all $0<i+j<u+v$, then $\pi
(x_{1}^{u}x_{2}^{v})=\beta _{u,v}(g_{1}^{u}g_{2}^{v}-1)$, i.e., $\pi
(x_{1}^{u}x_{2}^{v})$ is $(1,g_{1}^{u}g_{2}^{v})$-primitive. For by the
quantum binomial theorem \cite{Kassel}, for scalars $\omega(i,j)$,
\begin{equation*}
\Delta (x_{1}^{u}x_{2}^{v})=g_{1}^{u}g_{2}^{v}\otimes
x_{1}^{u}x_{2}^{v}+x_{1}^{u}x_{2}^{v}\otimes 1+\sum_{\substack{ 0\leq i\leq
u,0\leq j\leq v,  \\ 0<i+j<u+v}}\omega(i,j)
g_{1}^{i}x_{1}^{u-i}g_{2}^{j}x_{2}^{v-j}\otimes x_{1}^{i}x_{2}^{j}
\end{equation*}%
and applying $\pi \otimes \pi $ to this expression, we obtain that
\begin{equation*}
\Delta (\pi \left( x_{1}^{u}x_{2}^{v}\right) )=g_{1}^{u}g_{2}^{v}\otimes \pi
(x_{1}^{u}x_{2}^{v})+\pi (x_{1}^{u}x_{2}^{v})\otimes 1.
\end{equation*}
Using this argument with $0 <i+j = 1$ yields $\pi (x_{1}x_{2})=\beta
(g_{1}g_{2}-1)$ for some scalar $\beta$. We now test associativity on $%
x_{1},x_{2},x_{2}$. First we have that
\begin{eqnarray*}
x_{1}\cdot _{R}x_{2} &=& \tau (x_{1}x_{2}) =g_{1}g_{2}\sigma(S_{H}(\pi
(x_{1}x_{2})))+0+0+x_{1}x_{2} \\
&=& \beta g_{1}g_{2}((g_{1}g_{2})^{-1}-1)+x_{1}x_{2} =x_{1}x_{2}-\beta
(g_{1}g_{2}-1),
\end{eqnarray*}%
so that
\begin{equation*}
(x_{1}\cdot _{R}x_{2})\cdot _{R}x_{2} =\tau ( x_{1}x_{2}x_{2}-\beta
(g_{1}g_{2}-1)x_{2}) = \tau (x_{1}x_{2}^{2})-\beta (q ^{-2}-1)x_{2}.
\end{equation*}%
On the other hand, since by (\ref{ndifferentfromm}), $\pi (x_{2})=\pi
(x_{2}^{2})=0$, then
\begin{equation*}
x_{2}\cdot _{R}x_{2}=\tau (x_{2}^{2})=x_{2}^{2},
\end{equation*}%
and thus
\begin{equation*}
x_{1}\cdot _{R}(x_{2}\cdot _{R}x_{2})=x_{1}\cdot _{R}x_{2}^{2}=\tau
(x_{1}x_{2}^{2}).
\end{equation*}%
If $R$ is associative then these expressions must be equal and thus $\beta
=0 $. Now consider multiplication in $R$ of the elements $x_{2},x_{1},x_{1}$%
. First we compute
\begin{eqnarray*}
x_{2}\cdot _{R}x_{1} &=&\tau (x_{2}x_{1})=\tau (q
x_{1}x_{2}+a(g_{1}g_{2}-1)) = q \tau (x_{1}x_{2})=q x_{1}x_{2}\mbox{ since }%
\beta =0.
\end{eqnarray*}%
Thus $(x_{2}\cdot _{R}x_{1})\cdot _{R}x_{1}=\tau (q x_{1}x_{2}x_{1})$. On
the other hand
\begin{eqnarray*}
&& x_{2}\cdot _{R}(x_{1}\cdot _{R}x_{1}) = x_{2}\cdot _{R}x_{1}^{2}=\tau
(x_{2}x_{1}^{2}) \\
&=& \tau (q x_{1}x_{2}x_{1}+a(1-g_{1}g_{2})x_{1}) = \tau (q
x_{1}x_{2}x_{1})+a(q ^{2}-1)x_{1}.
\end{eqnarray*}%
This contradicts the choice of $a\neq 0$.\qed
\end{example}

\begin{remark}
In the above example, it is key that $r >2$ and $x_i^2 \neq 0$. In the
examples of dimension $32$ in Section \ref{sec: qlps}, $R$ is not thin, $\xi$
is nontrivial, but $R$ is a bialgebra in $_H^H\mathcal{YD}$ since the image
of $\xi$ lies in the centre of $A$.
\end{remark}

\vspace{2mm}

We now prove the converse to the observation in Remark \ref{rem:
associativity} in the case that $R$ is connected.

\begin{theorem}
\label{te: Rassoc} Let $(R ,\xi )$ be a pre-bialgebra with cocycle in ${%
_{H}^{H}\mathcal{YD}}$. If $R$ is connected, then the following are
equivalent.

\begin{itemize}
\item[(i)] $m_R$ is associative.

\item[(ii)] $\xi \left( z\right) t=\varepsilon \left( z\right) t,$ for every
$z\in R\otimes R,t\in R$.

\item[(iii)] $\Phi \left( \xi \right) =\mathrm{Id} _{R^{\otimes 3}}.$
\end{itemize}
\end{theorem}

\begin{proof}
By Remark \ref{re: Phi} and Remark \ref{rem: associativity} it remains only
to show that $(i)$ implies $(ii)$, i.e., to prove that if $m_R$ is
associative, then
\begin{equation}
\xi \left( r\otimes s\right) t=\varepsilon _{R}\left( r\right) \varepsilon
_{R}\left( s\right) t,\text{ for every }r,s,t\in R.
\label{form: trivial action}
\end{equation}
The argument is by induction on $u + v$ where $r\in R_u$ and $s \in R_v$.
For $u+v=0,1,$ then either $u=0$ or $v=0$ and by (\ref{eq:YD10'}), there is
nothing to show.

Since $R$ is connected, by \cite[Lemma 5.3.2, 2)]{Mo}, for every $n>0$ and $%
r\in R_{n}$ there exists a finite set $I$ and $r_{i},r^{i}\in R_{n-1},$ for
every $i\in I$, such that
\begin{equation*}
\Delta \left( r\right) =1_{R}\otimes r+r\otimes 1_{R}+\sum_{i\in
I}r_{i}\otimes r^{i},
\end{equation*}
and thus
\begin{equation}
\sum_{i\in I}r_{i}\varepsilon _{R}\left( r^{i}\right) =-\varepsilon \left(
r\right) 1_{R} = \sum_{i\in I}\varepsilon _{R}\left( r_{i}\right) r^{i}.
\end{equation}
Recall that by (\ref{form: DeltaAr})
\begin{eqnarray*}
\Delta _{A}\left( r\# 1_H \right) &=& r^{\left( 1\right) }\#\ r^{\left(
2\right) } _{\left\langle -1\right\rangle } \otimes r^{\left( 2\right) }
_{\left\langle 0\right\rangle }\#1_{H} \\
&=& (1_{R}\#r_{\left\langle -1\right\rangle } \otimes r_{\left\langle
0\right\rangle }\#1_{H}) + (r\#\ 1_{R} \otimes 1_{R} \#1_{H}) + \sum_{i\in
I}( r_{i}\#\ r^{i} _{\left\langle -1\right\rangle } \otimes r^{i}
_{\left\langle 0\right\rangle }\#1_{H}).
\end{eqnarray*}

Suppose that the statement holds for $u + v -1$ and let $r\in R_{u}$ with
comultiplication as above and $s\in R_{v} $ with $\Delta_R \left( s\right)
=1_{R}\otimes s+s\otimes 1_{R}+\sum_{j\in J}s_{j}\otimes s^{j} $. Let us
compute $r\cdot _{R}s.$ We have%
\begin{eqnarray*}
&&\left( r\#1_{H}\right) \left( s\#1_{H}\right) \overset{\left( \ref{form:
multi smash xi}\right)}{=} r^{\left( 1\right) }\cdot _{R}\left(
r_{\left\langle -1\right\rangle }^{\left( 2\right) }s^{\left( 1\right)
}\right) \#\xi \left( r_{\left\langle 0\right\rangle }^{\left( 2\right)
}\otimes s^{\left( 2\right) }\right) \\
&=&r_{\left\langle -1\right\rangle }s^{\left( 1\right) }\#\xi \left(
r_{\left\langle 0\right\rangle }\otimes s^{\left( 2\right) }\right) +r\cdot
_{R}s^{\left( 1\right) }\#\xi \left( 1_{R}\otimes s^{\left( 2\right)
}\right) +\sum_{i\in I}r_{i}\cdot _{R}\left( r_{\left\langle -1\right\rangle
}^{i}s^{\left( 1\right) }\right) \#\xi \left( r_{\left\langle 0\right\rangle
}^{i}\otimes s^{\left( 2\right) }\right) \\
&=& r\cdot _{R}s\#1_{H}+r_{\left\langle -1\right\rangle }s^{\left( 1\right)
}\#\xi \left( r_{\left\langle 0\right\rangle }\otimes s^{\left( 2\right)
}\right) +\sum_{i\in I}r_{i}\cdot _{R}\left( r_{\left\langle -1\right\rangle
}^{i}s^{\left( 1\right) }\right) \#\xi \left( r_{\left\langle 0\right\rangle
}^{i}\otimes s^{\left( 2\right) }\right) \\
&=&r\cdot _{R}s\#1_{H}+\left[
\begin{array}{c}
r_{\left\langle -1\right\rangle }1_{R}\#\xi \left( r_{\left\langle
0\right\rangle }\otimes s\right) + \\
+r_{\left\langle -1\right\rangle }s\# \xi \left( r_{\left\langle
0\right\rangle }\otimes 1_{R}\right) + \\
+\sum_{j\in J}r_{\left\langle -1\right\rangle }s_{j}\# \xi \left(
r_{\left\langle 0\right\rangle }\otimes s^{j}\right)%
\end{array}%
\right] +\left[
\begin{array}{c}
\sum_{i\in I}r_{i}\cdot _{R}\left( r_{\left\langle -1\right\rangle
}^{i}1_{R}\right) \#\xi \left( r_{\left\langle 0\right\rangle}^{i}\otimes
s\right) + \\
+\sum_{i\in I}r_{i}\cdot _{R}\left( r_{\left\langle -1\right\rangle
}^{i}s\right) \#\xi \left( r_{\left\langle 0\right\rangle }^{i}\otimes
1_{R}\right) + \\
+\sum_{\substack{ i\in I  \\ j\in J}}r_{i}\cdot _{R}\left( r_{\left\langle
-1\right\rangle }^{i}s_{j}\right) \#\xi \left( r_{\left\langle
0\right\rangle }^{i}\otimes s^{j}\right)%
\end{array}%
\right] \\
&=& r\cdot _{R}s\#1_{H}+\left[
\begin{array}{c}
1_{R}\#\xi \left( r\otimes s\right) + \\
+\varepsilon _{R}\left( r\right) s\#1_{H}+ \\
+\sum_{j\in J}r_{\left\langle -1\right\rangle }s_{j}\#\xi \left(
r_{\left\langle 0\right\rangle }\otimes s^{j}\right)%
\end{array}%
\right] +\left[
\begin{array}{c}
\sum_{i\in I}r_{i}\#\xi \left( r^{i}\otimes s\right) + \\
+\sum_{i\in I}r_{i}\varepsilon _{R}\left( r^{i}\right) \cdot _{R}s\#1_{H}+
\\
+\sum_{\substack{ i\in I  \\ j\in J}}r_{i}\cdot _{R}\left( r_{\left\langle
-1\right\rangle }^{i}s_{j}\right) \#\xi \left( r_{\left\langle
0\right\rangle }^{i}\otimes s^{j}\right)%
\end{array}%
\right]
\end{eqnarray*}

so that%
\begin{equation*}
\left( r\#1_{H}\right) \left( s\#1_{H}\right) =\left[
\begin{array}{c}
r\cdot _{R}s\#1_{H}+1_{R}\#\xi \left( r\otimes s\right) +\sum_{i\in
I}r_{i}\#\xi \left( r^{i}\otimes s\right) + \\
+\sum_{j\in J}r_{\left\langle -1\right\rangle }s_{j}\#\xi \left(
r_{\left\langle 0\right\rangle }\otimes s^{j}\right) +\sum_{\substack{ i\in
I  \\ j\in J}}r_{i}\cdot _{R}\left( r_{\left\langle -1\right\rangle
}^{i}s_{j}\right) \#\xi \left( r_{\left\langle 0\right\rangle }^{i}\otimes
s^{j}\right)%
\end{array}%
\right] .
\end{equation*}

Note that $\left( R\otimes \varepsilon _{H}\right) \left[ \left(
r\#1_{H}\right) \left( s\#h\right) \right] =r\cdot _{R}\left( R\otimes
\varepsilon _{H}\right) \left( s\#h\right) $ so that
\begin{equation*}
\left( R\otimes \varepsilon _{H}\right) \left[ \left( r\#1_{H}\right) \left(
s\#1_{H}\right) \left( t\#1_{H}\right) \right] =r\cdot _{R}\left( R\otimes
\varepsilon _{H}\right) \left[ \left( s\#1_{H}\right) \left( t\#1_{H}\right) %
\right] =r\cdot _{R}\left( s\cdot _{R}t\right)
\end{equation*}%
Then, we have%
\begin{eqnarray*}
0 &=&r\cdot _{R}\left( s\cdot _{R}t\right) -\left( r\cdot _{R}s\right) \cdot
_{R}t \\
&=&\left( R\otimes \varepsilon _{H}\right) \left[ \left( r\#1_{H}\right)
\left( s\#1_{H}\right) \left( t\#1_{H}\right) \right] -\left( R\otimes
\varepsilon _{H}\right) \left[ \left( r\cdot _{R}s\#1_{H}\right) \left(
t\#1_{H}\right) \right] \\
&=&\left( R\otimes \varepsilon _{H}\right) \left\{ \left[ \left(
r\#1_{H}\right) \left( s\#1_{H}\right) -r\cdot _{R}s\#1_{H}\right] \left(
t\#1_{H}\right) \right\} \\
&=&\left( R\otimes \varepsilon _{H}\right) \left[
\begin{array}{c}
\left( 1_{R}\#\xi \left( r\otimes s\right) \right) \left( t\#1_{H}\right) +
\\
+\sum_{i\in I}\left[ r_{i}\#\xi \left( r^{i}\otimes s\right) \right] \left(
t\#1_{H}\right) + \\
+\sum_{j\in J}\left[ r_{\left\langle -1\right\rangle }s_{j}\#\xi \left(
r_{\left\langle 0\right\rangle }\otimes s^{j}\right) \right] \left(
t\#1_{H}\right) + \\
+\sum_{\substack{ i\in I  \\ j\in J}}\left[ r_{i}\cdot _{R}\left(
r_{\left\langle -1\right\rangle }^{i}s_{j}\right) \#\xi \left(
r_{\left\langle 0\right\rangle }^{i}\otimes s^{j}\right) \right] \left(
t\#1_{H}\right)%
\end{array}%
\right].
\end{eqnarray*}

The first term in this sum is clearly $\xi(r \otimes s)t$ and it remains to
show that the other terms add to $-\varepsilon_R(r)\varepsilon_R(s)t$.

Since $r^i \in R_{u-1}$, the second term in the sum above is
\begin{eqnarray*}
&& \left( R\otimes \varepsilon _{H}\right)\left[\sum_{i\in I} (r_{i}\#\xi
\left( r^{i}\otimes s\right) )(t \# 1) \right] \\
&\overset{ (\ref{form: varepsilonHmA}) }{=}& \sum_{i \in I} r_i \cdot_R
(\xi(r^i \#s)t ) \\
&=& \sum_{i \in I} r_i \cdot_R (\varepsilon_R(r^i)\varepsilon_R(s)t) \text{
by the induction hypothesis} \\
&\overset{(\ref{form: trivial action})}{=}&
-\varepsilon_R(r)\varepsilon_R(s)t.
\end{eqnarray*}

A similar computation shows that $(R \otimes \varepsilon_H) (\sum_{j\in J}%
\left[ r_{\left\langle -1\right\rangle }s_{j}\#\xi \left( r_{\left\langle
0\right\rangle }\otimes s^{j}\right) \right] \left( t\#1_{H}\right)) =
-\varepsilon_R(r)\varepsilon_R(s)t$ and that $(R \otimes \varepsilon_H)(\sum
_{\substack{ i\in I  \\ j\in J}}\left[ r_{i}\cdot _{R}\left( r_{\left\langle
-1\right\rangle }^{i}s_{j}\right) \#\xi \left( r_{\left\langle
0\right\rangle }^{i}\otimes s^{j}\right) \right] \left( t\#1_{H}\right)) =
\varepsilon_R(r)\varepsilon_R(s)t$ so that
\begin{equation*}
0 =\xi \left( r\otimes s\right) t-\varepsilon _{R}\left( r\right)
\varepsilon _{H}\left( s\right) t-\varepsilon _{R}\left( r\right)
\varepsilon _{R}\left( s\right) t+\varepsilon _{H}\left( r\right)
\varepsilon _{H}\left( s\right) t = \xi \left( r\otimes s\right)
t-\varepsilon _{R}\left( r\right) \varepsilon _{H}\left( s\right) t
\end{equation*}
and the proof is finished.
\end{proof}

\section{Cocycle deformations of splitting data}

\label{sec: main}

Let $(A,H,\pi,\sigma)$ be a splitting datum with associated pre-bialgebra
with cocycle $(R,\xi)$. In this section, we extend the notion of a cocycle
deformation of $A$ to a cocycle deformation of $R$ and show how these are
related. For $\Gamma$ a finite abelian group, $V$ a crossed $k[\Gamma]$
module and $A = \mathcal{B}(V) \# K[\Gamma]$, then the results we present
should be compared to those in \cite[Section 4]{Grunenfelder-Mastnak}

Recall that if $A$ is a bialgebra, a convolution invertible map $\gamma
:A\otimes A\rightarrow K$ is called a unital (or normalized) 2-cocycle for $%
A $ when for all $x,y,z\in A$,
\begin{eqnarray}
\gamma \left( y_{\left( 1\right) }\otimes z_{\left( 1\right) }\right) \gamma
\left( x\otimes y_{\left( 2\right) }z_{\left( 2\right) }\right) &=&\gamma
\left( x_{\left( 1\right) }\otimes y_{\left( 1\right) }\right) \gamma \left(
x_{\left( 2\right) }y_{\left( 2\right) }\otimes z\right) ,
\label{form: cocycle1 Hopf} \\
\gamma (x\otimes 1) &=&\gamma (1\otimes x)=\varepsilon _{A}(x).
\label{form: cocycle3}
\end{eqnarray}

Note that (\ref{form: cocycle1 Hopf}) holds for all $x,y,z\in A$ if and only
if%
\begin{equation*}
(\varepsilon _{A}\otimes \gamma )\ast \gamma (A\otimes m_{A})=(\gamma
\otimes \varepsilon _{A})\ast \gamma (m_{A}\otimes A).
\label{form:cocycle1 HopfImplicit}
\end{equation*}
  For  a bialgebra $A$ with a subHopf algebra $H$, we denote
by $Z_{H}^{2}\left( A,K\right) $ the space of $H$-bilinear
$2$-cocycles for $A $, i.e., the set of cocycles as defined above
which are also $H$-bilinear.  If $H=K$ we write $Z^{2}(A,K) $
instead of $Z_H^{2}(A,K)$.

\par One may deform or twist $A$ by   any $\gamma \in Z^{2}(A,K)
$
  to get a new bialgebra $A^{\gamma }$%
. (See, for example, \cite[Theorem 1.6]{Doi-braided}.) As a coalgebra, $%
A^{\gamma }=A$, but the multiplication in $A^{\gamma }$ is given by
\begin{equation*}
x\cdot _{\gamma }y=x\cdot _{A^{\gamma }}y:=\gamma \left( x_{\left( 1\right)
}\otimes y_{\left( 1\right) }\right) x_{\left( 2\right) }y_{\left( 2\right)
}\gamma ^{-1}\left( x_{\left( 3\right) }\otimes y_{\left( 3\right) }\right) ,
\end{equation*}%
for all $x,y\in A$. By (\ref{form: cocycle3}), $A^{\gamma }$ has unit $1_{A}$
and condition (\ref{form: cocycle1 Hopf}) implies that the multiplication in
$A^{\gamma }$ is associative if and only if the multiplication in $A$ is
associative. If $A$ is a Hopf algebra with antipode $S$, by \cite[1.6(a5)]%
{Doi-braided} then $\gamma (x_{\left( 1\right) }\otimes S(x_{\left(
2\right) }))\gamma ^{-1}(S(x_{\left( 3\right) })\otimes x_{\left(
4\right) })=\varepsilon (x)$, so that $A^{\gamma }$ is also a Hopf
algebra with antipode given by
\begin{equation*}
S^{\gamma }(x)=\gamma (x_{\left( 1\right) }\otimes S(x_{\left( 2\right)
}))S(x_{\left( 3\right) })\gamma ^{-1}(S(x_{\left( 4\right) })\otimes
x_{\left( 5\right) }).
\end{equation*}

By \cite[1.6(a3)]{Doi-braided}, $\gamma ^{-1}$ is a cocycle for $A^{cop}$
and, as algebras, $A^{\gamma }=(A^{cop})^{\gamma ^{-1}}$.

\begin{definition}
{Let $A$ be a bialgebra and let $\beta ,\gamma :A\otimes A\rightarrow K$ be
  $K$-bilinear maps. Denote by }$_{\beta }A_{\gamma }$ the vector space $A$
endowed with the following not necessarily associative multiplication%
\begin{equation*}
x\cdot _{_{\beta }A_{\gamma }}y=\beta (x_{(1)}\otimes
y_{(1)})x_{(2)}y_{(2)}\gamma (x_{(3)}\otimes y_{(3)}),\text{ for all }x,y\in
A\text{.}
\end{equation*}
\end{definition}

\begin{remark}
\label{re: cocycle iff assoc} For $A,\gamma,\beta$ as above, if
$\gamma
  =\varepsilon_{A\otimes A} $, then we denote ${_{\beta }A_{\gamma
  }}$ simply by
  ${_{\beta }A}$. The multiplication of $_{\beta }A $ is just
denoted by $\ast _{\beta }$ where   $x\ast _{\beta }y=\beta
(x_{(1)}\otimes y_{(1)})x_{(2)}y_{(2)}.$  Then  $\beta$ satisfies
 \eqref{form: cocycle1 Hopf} and \eqref{form: cocycle3}  if and only
if $  (A,\ast _{\beta })$ is an associative algebra with
$1_{A}=1_{(A,\ast _{\beta }) }$. The condition on $1_A$ is
equivalent to $(\ref{form: cocycle3})$. The associativity
statement follows from computing  \begin{eqnarray*}
\lefteqn{(x\ast _{\beta} y)\ast _{\beta} z=\beta (x_{(1)}\otimes
y_{(1)})\beta
(x_{(2)}y_{(2)}\otimes z_{(1)})x_{(3)}y_{(3)}z_{(2)}} \\
&=&[(\beta \otimes \varepsilon _{A})\ast \beta (m_{A}\otimes
A)](x_{(1)}\otimes y_{(1)}\otimes z_{(1)})x_{(2)}y_{(2)}z_{(2)}
\end{eqnarray*}%
and
\begin{eqnarray*}
\lefteqn{x\ast _{\beta}(y\ast _{\beta} z)=\beta (y_{(1)}\otimes
z_{(1)})\beta
(x_{(1)}\otimes y_{(2)}z_{(2)})x_{(2)}y_{(3)}z_{(3)}} \\
&=&[(\varepsilon _{A}\otimes \beta )\ast \beta (A\otimes
m_{A})](x_{(1)}\otimes y_{(1)}\otimes
z_{(1)})x_{(2)}y_{(2)}z_{(2)}
\end{eqnarray*} %
Thus clearly if $\beta $ satisfies (\ref{form: cocycle1 Hopf}),
then $\ast _{\beta }$ is an associative operation, and, applying
$\varepsilon $ to the expressions above, we see that the converse
holds.

Similarly,  if $\beta =\varepsilon _{A\otimes A},$
  we denote
$A_{\gamma }:={_{\beta }A_{\gamma }}.$  The multiplication of $
A_{\gamma }$ will be simply denoted by $\ast ^{\gamma }$  so that
it is defined by $x\ast ^{\gamma }y=x_{(1)}y_{(1)}\gamma
(x_{(2)}\otimes y_{(2)}).$  Then $\ast
^{\gamma }$ is an associative operation if and only if $\gamma $ satisfies (%
\ref{form: cocycle1 Hopf}) for $A^{cop}$. Then   if $A$ is a
bialgebra, $\gamma$ satisfies (\ref{form: cocycle1 Hopf}) and
(\ref{form: cocycle3}) for $A^{cop}$ if and only if $A_{\gamma
}:=(A,\ast ^{\gamma })$ is associative with unit $1_A$.
\par Observe that, for   $\gamma \in Z^{2}(A,K),$
 one has $A^{\gamma
}={_{\gamma }A_{\gamma ^{-1}}}$ as an algebra. \qed
\end{remark}

 The next lemma will be useful in building examples in
the last section of this paper.
\begin{lemma}
\label{lm: building cocycles} For $A$ a bialgebra, let $\beta
,\gamma :A\otimes A\rightarrow K$ be $K$-bilinear convolution
invertible maps. Suppose that $_{\beta }A_{\gamma }$ is an
associative unitary algebra.   Then $\beta\in  Z^{2}(A,K)$ if and
only if $\gamma ^{-1}\in  Z^{2}(A,K)$.
\end{lemma}

\begin{proof}
 For any $K$-bilinear map $\sigma: A\otimes
A\rightarrow K,$  we define maps $X(\sigma),Y(\sigma): A \otimes A
\otimes A \rightarrow K$ by
\begin{eqnarray*}
&& X\left( \sigma \right) \left( a\otimes b\otimes c\right) :=\sigma
(a_{(1)}\otimes b_{(1)})\sigma (a_{(2)}b_{(2)}\otimes c), \\
&& Y\left( \sigma \right) \left( a\otimes b\otimes c\right) :=\sigma
(b_{(1)}\otimes c_{(1)})\sigma (a\otimes b_{(2)}c_{(2)}),
\end{eqnarray*}
 for all $a,b,c \in A$.
Thus $\sigma $ {satisfies (\ref{form: cocycle1 Hopf}) if and only if }$%
X\left( \sigma \right) =Y\left( \sigma \right) .$
We have
\begin{eqnarray*}
\lefteqn{(a\cdot_{_\beta A_\gamma} b)\cdot_{_\beta A_\gamma} c} \\
&=&(\beta (a_{(1)}\otimes b_{(1)})a_{(2)}b_{(2)}\gamma (a_{(3)}\otimes
b_{(3)}))\cdot_{_\beta A_\gamma} c \\
&=&\beta (a_{(1)}\otimes b_{(1)})\beta (a_{(2)}b_{(2)}\otimes
c_{(1)})a_{(3)}b_{(3)}c_{(2)}\gamma (a_{(4)}b_{(4)}\otimes c_{(3)})\gamma
(a_{(5)}\otimes b_{(5)}) \\
&=&X\left( \beta \right) \left( a_{(1)}\otimes b_{(1)}\otimes c_{(1)}\right)
a_{(2)}b_{(2)}c_{(2)}\left[ X\left( \gamma ^{-1}\right) \right] ^{-1}\left(
a_{(3)}\otimes b_{(3)}\otimes c_{(3)}\right) ,
\end{eqnarray*}
and
\begin{eqnarray*}
\lefteqn{ a  \cdot_{_\beta A_\gamma}  (b\cdot_{_\beta A_\gamma} c)} \\
&=&a\cdot_{_\beta A_\gamma} (\beta (b_{(1)}\otimes
c_{(1)})b_{(2)}c_{(2)}\gamma
(b_{(3)}\otimes c_{(3)})) \\
&=&\beta (b_{(1)}\otimes c_{(1)})\beta (a_{(1)}\otimes
b_{(2)}c_{(2)})a_{(2)}b_{(3)}c_{(3)}\gamma (a_{(3)}\otimes
b_{(4)}c_{(4)})\gamma (b_{(5)}\otimes c_{(5)}) \\
&=&Y\left( \beta \right) \left( a_{(1)}\otimes b_{(1)}\otimes c_{(1)}\right)
a_{(2)}b_{(2)}c_{(2)}\left[ Y\left( \gamma ^{-1}\right) \right] ^{-1}\left(
a_{(3)}\otimes b_{(3)}\otimes c_{(3)}\right)
\end{eqnarray*}%
{where }$\left[ X\left( \gamma ^{-1}\right) \right] ^{-1}$ and
$\left[ Y\left( \gamma ^{-1}\right) \right] ^{-1}$ denote the
convolution inverses of $X\left( \gamma ^{-1}\right) $ and $Y\left(
\gamma ^{-1}\right) $ respectively. Since $(a\cdot_{_\beta A_\gamma}
b)\cdot_{_\beta A_\gamma} c=a\cdot_{_\beta A_\gamma} (b\cdot_{_\beta
A_\gamma} c), $ by applying $\varepsilon _{A}$ to both sides we
obtain
\begin{equation*}
X\left( \beta \right) \ast \left[ X\left( \gamma ^{-1}\right) \right]
^{-1}=Y\left( \beta \right) \ast \left[ Y\left( \gamma ^{-1}\right) \right]
^{-1}
\end{equation*}%
that is,
\begin{equation*}
\left[ X\left( \beta \right) \right] ^{-1}\ast Y\left( \beta \right) =\left[
X\left( \gamma ^{-1}\right) \right] ^{-1}\ast Y\left( \gamma ^{-1}\right) .
\end{equation*}%
It is now clear that {$\beta $ satisfies (\ref{form: cocycle1 Hopf}) if and
only if }$\gamma ^{-1}$ does.

We have%
\begin{equation*}
b=1\cdot_{_\beta A_\gamma} b=\beta (1_{(1)}\otimes
b_{(1)})1_{(2)}b_{(2)}\gamma (1_{(3)}\otimes b_{(3)})
\end{equation*}%
so that
\begin{equation*}
b=\beta (1\otimes b_{(1)})b_{(2)}\gamma (1\otimes b_{(3)}).
\end{equation*}%
By applying $\varepsilon _{A}$ to both sides, we obtain $\varepsilon
_{A}\left(
b\right) =\beta (1\otimes b_{(1)})\gamma (1\otimes b_{(2)})$ which yields%
\begin{equation*}
\beta (1\otimes -)=\gamma ^{-1}(1\otimes -).
\end{equation*}%
Similarly $a=a\cdot_{_\beta A_\gamma} 1$ yields $\beta (-\otimes
1)=\gamma ^{-1}(-\otimes
1).$ Therefore {$\beta $ satisfies (\ref{form: cocycle3}) if and only if }$%
\gamma ^{-1}$ does.
\end{proof}

 \vspace{1mm}

\begin{corollary}
\label{co: beattie} For $A$ a  bialgebra,  let $\beta \in
Z^{2}(A,K)$ and  $\gamma \in Z^{2}(A^{\beta },K).$ Then $\gamma \ast
\beta \in Z^{2}(A,K).$
\end{corollary}

\begin{proof}
By {Remark \ref{re: cocycle iff assoc}, }$_{\gamma }\left( A^{\beta
}\right) $ is associative. Now $_{\gamma }\left( A^{\beta }\right)
=${$_{\gamma \ast
\beta }A_{\beta ^{-1}}$ }so that, {by Lemma \ref{lm: building cocycles}, }$%
\gamma \ast \beta ${$\in Z^{2}(A,K)$. }
\end{proof}
\vspace{2mm}

 A map $\gamma\in Hom(A
\otimes A,K)$ is called $H$-balanced if $\gamma : A \otimes_H A
\rightarrow K$, in other words, for all $a,a^{\prime }\in A, h \in
H$,
\begin{equation}
\gamma(a\sigma(h) \otimes a^{\prime}) = \gamma(a \otimes
\sigma(h)a^{\prime}).
\end{equation}

\begin{lemma}
\label{lem: balanced}Let $A$ be a bialgebra, $H$ a Hopf algebra and
$\sigma: H \rightarrow A$ a bialgebra monomorphism. Let $\gamma \in
Z^2_H(A,K)$. Then

\begin{enumerate}
\item[(i)] $\gamma$ is $H$-balanced.
\item[(ii)] $\gamma^{-1}$ is also $H$-bilinear and $H$-balanced.
\end{enumerate}
\end{lemma}

\begin{proof}
\par (i) By applying (\ref{form: cocycle1 Hopf}) with $y=\sigma(h) $ and
using $H$-bilinearity of $\gamma$, we get that $\gamma $ is
$H$-balanced.

\par (ii)  For $a,a^{\prime }\in A,h,h^{\prime }\in H$, we have%
\begin{eqnarray*}
&&\gamma ^{-1}\left( \sigma (h)a\otimes a^{\prime }\sigma (h^{\prime
})\right) \\
&=&\gamma ^{-1}\left( \sigma (h)a_{\left( 1\right) }\otimes
a_{\left( 1\right) }^{\prime }\sigma (h^{\prime })\right) \gamma
\left( a_{\left( 2\right) }\otimes a_{\left( 2\right) }^{\prime
}\right) \gamma ^{-1}\left(
a_{\left( 3\right) }\otimes a_{\left( 3\right) }^{\prime }\right) \\
&=&\gamma ^{-1}\left( \sigma (h)_{\left( 1\right) }a_{\left(
1\right) }\otimes a_{\left( 1\right) }^{\prime }\sigma (h^{\prime
})_{\left( 1\right) }\right) \gamma \left( \sigma (h)_{\left(
2\right) }a_{\left( 2\right) }\otimes a_{\left( 2\right) }^{\prime
}\sigma (h)_{\left( 2\right) }^{\prime }\right) \gamma ^{-1}\left(
a_{\left( 3\right) }\otimes a_{\left( 3\right)
}^{\prime }\right) \\
&=&\varepsilon _{H}\left( h\right) \gamma ^{-1}\left( a\otimes
a^{\prime }\right) \varepsilon _{H}\left( h^{\prime }\right) ,
\end{eqnarray*}%
and so $\gamma ^{-1}$ is $H$-bilinear. Similarly, write $\gamma
^{-1}(a\sigma (h)\otimes a^{\prime })$ as $\gamma ^{-1}(a_{1}\sigma
(h_{1})\otimes a_{1}^{\prime })\gamma (a_{2}\otimes \sigma
(h_{2})a_{2}^{\prime })\gamma ^{-1}(a_{3}\otimes \sigma
(h_{3})a_{3}^{\prime })$, to see that that $\gamma ^{-1}$ is
$H$-balanced.
\end{proof}

\vspace{1mm}
\begin{lemma}
\label{re: splitting} Let $(A,H,\pi ,\sigma )$ be a splitting datum and let $%
\gamma \in Z_{H}^{2}(A,K)$. Then $(A^{\gamma },H,\pi ,\sigma )$ is also a
splitting datum with $A^{co\pi }={A^{\gamma }}^{co\pi }$ as coalgebras in $%
_{H}^{H}\mathcal{YD}$.
\end{lemma}

\begin{proof}
Since $A^{\gamma }=A$ as coalgebras, in order to prove that $(A^{\gamma
},H,\pi ,\sigma )$ is a splitting datum we have to check that $\sigma $ is
an algebra homomorphism and that $\pi $ is $H$-bilinear. Since both $\gamma $
and $\gamma ^{-1}$ are $H$-bilinear, for $h,h^{\prime }\in H$ and $a\in A,$
we get%
\begin{eqnarray*}
\sigma \left( h\right) \cdot _{\gamma }a &=&\gamma \left( \sigma \left(
h_{\left( 1\right) }\right) \otimes a_{\left( 1\right) }\right) \sigma
\left( h_{\left( 2\right) }\right) a_{\left( 2\right) }\gamma ^{-1}\left(
\sigma \left( h_{\left( 3\right) }\right) \otimes a_{\left( 3\right)
}\right)  \\
&=&\gamma \left( 1_{A}\otimes a_{\left( 1\right) }\right) \sigma \left(
h\right) a_{\left( 2\right) }\gamma ^{-1}\left( 1_{A}\otimes a_{\left(
3\right) }\right) =\sigma \left( h\right) a.
\end{eqnarray*}%
Similarly $a\cdot _{\gamma }\sigma \left( h\right) =a\sigma \left( h\right) .
$ Thus $\sigma \left( h\right) \cdot _{\gamma }\sigma \left( h^{\prime
}\right) =\sigma \left( h\right) \sigma \left( h^{\prime }\right) =\sigma
\left( hh^{\prime }\right) $ and $\pi \left( \sigma \left( h\right) \cdot
_{\gamma }a\cdot _{\gamma }\sigma \left( h^{\prime }\right) \right) =\pi
\left( \sigma \left( h\right) a\sigma \left( h^{\prime }\right) \right)
=h\pi \left( a\right) h^{\prime }$ for all $h,h^{\prime }\in H$ and $a\in A.$
Hence $(A^{\gamma },H,\pi ,\sigma )$ is a splitting datum. The corresponding
map $\tau _{\gamma }:A^{\gamma }\rightarrow R,$ as in \ref{sec: A co H} is
given by
\begin{equation*}
\tau _{\gamma }\left( a\right) =a_{\left( 1\right) }\cdot _{\gamma }\sigma
S\pi \left( a_{\left( 2\right) }\right) =a_{\left( 1\right) }\sigma S\pi
\left( a_{\left( 2\right) }\right) =\tau \left( a\right) .
\end{equation*}%
Using this fact, the last part of the statement follows by definition of the
coalgebra structures of $A^{co\pi }$ and ${A^{\gamma }}^{co\pi }$ in $%
_{H}^{H}\mathcal{YD}$ as given in \ref{sec: A co H}.
\end{proof}

\vspace{1mm}  Now we offer an appropriate definition for a
2-cocycle $  \upsilon :R\otimes R\rightarrow K$.

\begin{definition}
\label{def: Rcocycle} A convolution invertible map $\upsilon: R
\otimes R \rightarrow K$ (where $R\otimes R$ has the coalgebra
structure in \ref{sec: ydcoalgebras}) is called a unital 2-cocycle
for $(R,\xi)$ if for $\Phi(\xi) \in \mathrm{{End}(R \otimes R
\otimes R)}$ from Lemma \ref{lem: Phi},
\begin{eqnarray}
\left( \varepsilon _{R}\otimes \upsilon \right) \ast \upsilon \left(
R\otimes m_{R}\right) &=&\left( \upsilon \otimes \varepsilon
_{R}\right) \ast \left\{ \upsilon \left( m_{R}\otimes R\right) \Phi
\left( \xi \right)
\right\}, \text{ and }  \label{form: cocycle1} \\
\upsilon \left( - \otimes 1_{H}\right) &=&\varepsilon _{H} =\upsilon
\left( 1_{H}\otimes -\right).  \label{form: cocycle2}
\end{eqnarray}
\end{definition}

\vspace{2mm}
\par We will denote by  $Z^2_H(R,K)$   the space of left $H$-linear
2-cocycles for $R$. \vspace{1mm}
\par Given $\upsilon \in Z_H^{2}\left( R,K\right)$, let $R^{\upsilon
}$ be the
coalgebra $R \in {_{H}^{H}\mathcal{YD}}$ with multiplication defined by%
\begin{equation*}
m_{R^{\upsilon }}:=\left( \upsilon \otimes m_{R}\otimes \upsilon
^{-1}\right) \Delta _{R\otimes R}^{2} \text{ and unit } u_{R^\upsilon} = u_R.
\end{equation*}

We will see in Theorem \ref{teo: smash gamma} that $R^\upsilon$ is also a
pre-bialgebra with cocycle.

\vspace{2mm}

\begin{lemma}
\label{lem:RA} Let $(R, \xi)$ be a pre-bialgebra with cocycle and $(A = R
\#_\xi H,H,\pi,\sigma)$ be the associated splitting datum. Let $\phi: A
\otimes A \rightarrow K$ be $H$-bilinear and $H$-balanced. Then for $r,s,t
\in R$, $h \in H$, the following hold.
\begin{eqnarray}
\phi \left[ r\#1_{H}\otimes \left( s\#1_{H}\right) \left( t\#1_{H}\right) %
\right] &=& \phi \left( r\#1_{H}\otimes st \#1_{H}\right) ,
\label{form: utile gamma 2} \\
\phi \left( r\#h\otimes s\#1_{H}\right)&=& \phi \left( r\#1_{H}\otimes
hs\#1_{H}\right) ,  \label{form: utile gamma} \\
\phi(hr \# 1_H \otimes s \# 1_H)&=& \phi(r \# 1_H \otimes S(h) s \# 1_H).
\label{form: utile gamma 1}
\end{eqnarray}
\end{lemma}

\begin{proof}
The first statement holds since, using right $H$-linearity at the second
step,
\begin{eqnarray*}
\phi \left[ r\#1_{H}\otimes \left( s\#1_{H}\right) \left( t\#1_{H}\right) %
\right] &=&\phi \left[ r\#1_{H}\otimes \left( m_{R}\otimes \xi \right)
\Delta _{R\otimes R}\left( s\otimes t\right) \right] \\
&=&\phi \left[ r\#1_{H}\otimes \left( m_{R}\otimes u_H \varepsilon _{H}\xi
\right) \Delta _{R\otimes R}\left( s\otimes t\right) \right] \\
&=&\phi \left( r\#1_{H}\otimes st\#1_{H}\right)
\end{eqnarray*}%
The second equation follows from:
\begin{eqnarray*}
\phi \left( r\#h\otimes s\#1_{H}\right) &=&\phi \left[ \left(
r\#1_{H}\right) \left( 1_{R}\#h\right) \otimes s\#1_{H}\right] \\
&=&\phi \left[ r\#1_{H}\otimes \left( 1_{R}\#h\right) \left( s\#1_{H}\right) %
\right] \text{ ($\phi$ $H$-balanced)} \\
&=&\phi \left( r\#1_{H}\otimes h_{\left( 1\right) }s\#h_{\left( 2\right)
}\right) \\
&=&\phi \left[ r\#1_{H}\otimes \left( h_{\left( 1\right) }s\#1_{H}\right)
\left( 1_{R}\#h_{\left( 2\right) }\right) \right] \\
&=&\phi \left( r\#1_{H}\otimes h_{\left( 1\right) }s\#1_{H}\right)
\varepsilon _{H}\left( h_{\left( 2\right) }\right) \text{($\phi$ $H$%
-bilinear)} \\
&=&\phi \left( r\#1_{H}\otimes hs\#1_{H}\right)
\end{eqnarray*}%
Finally we check (\ref{form: utile gamma 1}).
\begin{eqnarray*}
\phi \left( hr\#1_{H}\otimes s\#1_{H}\right) &=&\phi \left[ \left(
1_{R}\#h_{\left( 1\right) }\right) \left( r\#S\left( h_{\left( 2\right)
}\right) \right) \otimes s\#1_{H}\right] \\
&=&\varepsilon _{H}\left( h_{\left( 1\right) }\right) \phi \left[ \left(
r\#S\left( h_{\left( 2\right) }\right) \right) \otimes s\#1_{H}\right] \text{%
($H$-balanced)} \\
&=&\phi \left( r\#S\left( h\right) \otimes s\#1_{H}\right) \\
&\overset{(\ref{form: utile gamma})}{=}&\phi \left( r\#1_{H}\otimes S\left(
h\right) s\#1_{H}\right) .
\end{eqnarray*}
\end{proof}

Now let $\mathcal{BB}(A)$ denote the set of $H$-bilinear $H$-balanced maps
from $A \otimes A$ to $K$ and $\mathcal{L}(R)$ the set of left $H$-linear
maps from $R \otimes R$ to $K$. The next proposition sets the stage for our
first theorem.

\begin{proposition}
Let $(R, \xi)$ be a pre-bialgebra with cocycle and $(A = R \#_\xi
H,H,\pi,\sigma)$ be the associated splitting datum. There is a bijective
correspondence between $\mathcal{BB}(A)$ and $\mathcal{L }(R)$ given by:
\begin{eqnarray*}
&& \Omega: Hom(A \otimes A ,K) \rightarrow Hom(R \otimes R,K) \text{ by }
\gamma \mapsto \gamma _{R} :=\gamma _{\mid R\otimes R} \text{ with inverse}
\\
&& \Omega^{\prime }: Hom(R \otimes R,K)\rightarrow Hom(A \otimes A,K) \text{
by } \upsilon \mapsto \upsilon_A:= \upsilon\circ(R \otimes \mu \otimes
\varepsilon_H ), \\
\hspace{1mm} && \text{ so that } \upsilon_A \left( x\#h\otimes y\#h^{\prime
}\right) =\upsilon \left( x\otimes hy\right) \varepsilon _{H}\left(
h^{\prime }\right) \text{ for each } h\in H,r,s\in R.
\end{eqnarray*}

Furthermore $\mathcal{BB}(A)$ and $\mathcal{L}(R)$ are both closed under the
convolution product and $\Omega$ and $\Omega^{\prime }$ preserve convolution.
\end{proposition}

\begin{proof}
Let $\gamma \in \mathcal{BB}(A)$ and we wish to show that $\Omega (\gamma
)=\gamma _{R}$ is in $\mathcal{L}(R)$. By (\ref{form: utile gamma 1}), we
have $\gamma _{R}\left( hr\otimes s\right) =\gamma _{R}\left( r\otimes
S\left( h\right) s\right) $ and thus
\begin{equation}
\gamma _{R}\left( h_{(1)}r\otimes h_{(2)}s\right) =\gamma _{R}\left(
r\otimes S(h_{(1)})h_{(2)}s\right) = \varepsilon_H (h)\gamma _{R}(r\otimes
s),  \notag
\end{equation}
and $\gamma_R$ is left $H$-linear. Conversely suppose $\upsilon \in \mathcal{%
L}(R)$ and check that $\Omega ^{\prime }(\upsilon )=\upsilon _{A}$ is $H$%
-bilinear. For $h,h^{\prime },l,m\in H$ and $x,y\in R$,
\begin{eqnarray*}
\upsilon _{A}\left[ \left( 1_{R}\#l\right) \left( x\#h\right) \otimes \left(
y\#h^{\prime }\right) \left( 1_{R}\#m\right) \right] &=&\upsilon _{A}\left(
l_{\left( 1\right) }x\#l_{\left( 2\right) }h\otimes y\#h^{\prime }m\right) \\
&\overset{\text{defn}}{=}&\upsilon \left( l_{\left( 1\right) }x\otimes
l_{\left( 2\right) }hy\right) \varepsilon _{H}\left( h^{\prime }m\right) \\
&\overset{}{=}&\varepsilon _{H}\left( l\right) \upsilon \left(
x\otimes hy\right) \varepsilon _{H}\left( h^{\prime }m\right)
\\
&=&\varepsilon _{H}\left( l\right) \upsilon _{A}\left( x\#h\otimes
y\#h^{\prime }\right) \varepsilon _{H}\left( m\right) .
\end{eqnarray*}
The fact that $\upsilon_A$ is $H$-balanced follows directly from the
definition.

For $r,s\in R$ and $h,m\in H$, we have that for $\gamma \in \mathcal{BB}(A)$%
,
\begin{equation}
[\Omega ^{\prime }(\gamma _{R})](r\#h\otimes s\#m) =\gamma _{R}(r\otimes hs)
\varepsilon_H (m) =\gamma (r\#1\otimes hs\#m)=\gamma (r\#h\otimes s\#m),
\notag
\end{equation}
and for $\upsilon \in \mathcal{L}(R)$,
\begin{equation}
[\Omega (\upsilon _{A})](r\otimes s) =\upsilon _{A}(r\#1\otimes
s\#1)=\upsilon (r\otimes s).  \notag
\end{equation}

Thus $\Omega $ and $\Omega ^{\prime }$ are inverse bijections. For $\gamma
,\gamma ^{\prime }\in \mathcal{BB}(A)$, it is clear that $\gamma \ast \gamma
^{\prime }$ is $H$-bilinear and $H$-balanced. Also for $\upsilon ,\upsilon
^{\prime }\in \mathcal{L}(R)$, $h\in H$, $r,s\in R$,
\begin{eqnarray*}
(\upsilon \ast \upsilon ^{\prime })(h_{(1)}r\otimes h_{(2)}s) &=&\upsilon
(h_{(1)}r^{(1)}\otimes (h_{(2)}r^{(2)})_{\langle -1\rangle
}h_{(3)}s^{(1)})\upsilon ^{\prime }((h_{(2)}r^{2})_{\langle 0\rangle
}\otimes h_{(4)}s^{(2)}) \\
&=&\upsilon (h_{(1)}r^{(1)}\otimes h_{(2)}r_{\langle -1\rangle
}^{(2)}S(h_{(4)})h_{(5)}s^{(1)})\upsilon ^{\prime }(h_{(3)}r_{\langle
0\rangle }^{2}\otimes h_{(6)}s^{(2)}) \\
&=&\varepsilon _{H}(h)\upsilon (r^{(1)}\otimes r_{\langle -1\rangle
}^{(2)}s^{(1)})\upsilon ^{\prime }(r_{\langle 0\rangle }^{(2)}\otimes
s^{(2)}) \\
&=&\varepsilon _{H}(h)(\upsilon \ast \upsilon ^{\prime })(r\otimes s).
\end{eqnarray*}%
Thus $\mathcal{BB}(A)$ and $\mathcal{L}(R)$ are closed under convolution and
it remains to show that $\Omega ,\Omega ^{\prime }$ are convolution
preserving. First we let $\gamma ,\gamma ^{\prime }\in \mathcal{BB}(A)$ and
we check that $\gamma _{R}\ast \gamma _{R}^{\prime }=(\gamma \ast \gamma
^{\prime })_{R}$. For every $x,y\in R,$ we have
\begin{eqnarray*}
\left( \gamma _{R}\ast \gamma _{R}^{\prime }\right) \left( x\otimes y\right)
&=&\gamma _{R}\left[ \left( x\otimes y\right) ^{\left( 1\right) }\right]
\gamma _{R}^{\prime }\left[ \left( x\otimes y\right) ^{\left( 2\right) }%
\right] \\
&=&\gamma _{R}\left( x^{\left( 1\right) }\otimes x_{\left\langle
-1\right\rangle }^{\left( 2\right) }y^{\left( 1\right) }\right) \gamma
_{R}^{\prime }\left( x_{\left\langle 0\right\rangle }^{\left( 2\right)
}\otimes y^{\left( 2\right) }\right) \\
&=&\gamma \left( x^{\left( 1\right) }\#1_{H}\otimes x_{\left\langle
-1\right\rangle }^{\left( 2\right) }y^{\left( 1\right) }\#1_{H}\right)
\gamma ^{\prime }\left( x_{\left\langle 0\right\rangle }^{\left( 2\right)
}\#1_{H}\otimes y^{\left( 2\right) }\#1_{H}\right) \\
&\overset{(\ref{form: utile gamma})}{=}&\gamma \left[ x^{\left( 1\right)
}\#x_{\left\langle -1\right\rangle }^{\left( 2\right) }\otimes y^{\left(
1\right) }\#1_{H}\right] \gamma ^{\prime }\left( x_{\left\langle
0\right\rangle }^{\left( 2\right) }\#1_{H}\otimes y^{\left( 2\right)
}\#1_{H}\right) \\
&=&\gamma \left[ x^{\left( 1\right) }\#x_{\left\langle -1\right\rangle
}^{\left( 2\right) }\otimes y^{\left( 1\right) }\#y_{\left\langle
-1\right\rangle }^{\left( 2\right) }\right] \gamma ^{\prime }\left(
x_{\left\langle 0\right\rangle }^{\left( 2\right) }\#1_{H}\otimes
y_{\langle0\rangle}^{\left( 2\right) }\#1_{H} \right) \\
&=&\gamma \left[ \left( x\#1_{H}\right) _{\left( 1\right) }\otimes \left(
y\#1_{H}\right) _{\left( 1\right) }\right] \gamma ^{\prime }\left( \left(
x\#1_{H}\right) _{\left( 2\right) }\otimes \left( y\#1_{H}\right) _{\left(
2\right) }\right) \\
&=&\left( \gamma \ast \gamma ^{\prime }\right) \left( x\#1_{H}\otimes
y\#1_{H}\right) =\left( \gamma \ast \gamma ^{\prime }\right) _{R}(x\otimes
y).
\end{eqnarray*}%
Finally, to see that $\Omega ^{\prime }(\upsilon \ast \upsilon ^{\prime
})=\Omega ^{\prime }(\upsilon )\ast \Omega ^{\prime }(\upsilon ^{\prime })$,
apply $\Omega $ to both sides and use the fact that $\Omega $ is one-one and
convolution preserving.
\end{proof}

In fact, $\Omega$ maps 2-cocycles to 2-cocycles.

\begin{theorem}
\label{te:main1} Let $(R,\xi )$ be a pre-bialgebra with cocycle with $%
(A=R\#_{\xi }H,H,\pi ,\sigma )$ the associated splitting datum. Then $\Omega
$ and $\Omega ^{\prime }$ as above are inverse bijections between $%
Z_{H}^{2}\left( A,K\right) $ and $Z_{H}^{2}\left( R,K\right) $.
\end{theorem}

\begin{proof}
First we note that clearly $\Omega, \Omega^{\prime }$ preserve the normality
conditions $(\ref{form: cocycle3})$ and $(\ref{form: cocycle2})$. It remains
to show that $\Omega, \Omega^{\prime }$ are compatible with the cocycle
conditions (\ref{form: cocycle1 Hopf}) and (\ref{form: cocycle1}).

Let $\upsilon \in Z_{H}^{2}\left( R,K\right) $. We will show that for $%
x,y,z\in R$, $h,h^{\prime },h^{\prime \prime }\in H$, then the left (right)
hand side of (\ref{form: cocycle1 Hopf}) for $\gamma =\upsilon _{A}$ acting
on $x\#h\otimes y\#h^{\prime }\otimes z\#h^{\prime \prime }$ equals the left
(right) hand side of (\ref{form: cocycle1}) applied to $(x\otimes
h_{1}y\otimes h_{2}h^{\prime }z)\varepsilon (h^{\prime \prime })$. Thus $%
\upsilon _{A}$ satisfies (\ref{form: cocycle1 Hopf}) if and only if $%
\upsilon $ satisfies (\ref{form: cocycle1}). (To see the \textquotedblleft
if\textquotedblright\ implication, let $1=h=h^{\prime }=h^{\prime \prime }$%
.) We start with the left hand side of $(\ref{form: cocycle1 Hopf})$ with $%
\gamma =\upsilon _{A}$.
\begin{eqnarray}
&& (\varepsilon _{A}\otimes \upsilon _{A} )\ast \upsilon _{A}(A\otimes
m_{A})(x\#h\otimes y\#h^{\prime }\otimes z\#h^{\prime \prime })
\label{eq: lhs} \\
&=&\upsilon _{A}\left( \left( y\#h^{\prime }\right) _{\left( 1\right)
}\otimes \left( z\#h^{\prime \prime }\right) _{\left( 1\right) }\right)
\upsilon _{A}\left( x\#h\otimes \left( y\#h^{\prime }\right) _{\left(
2\right) }\left( z\#h^{\prime \prime }\right) _{\left( 2\right) }\right)
\notag \\
&\overset{(\ref{eq: DeltaA})}{=}&\upsilon _{A}\left( y^{\left( 1\right)
}\#y_{\left\langle -1\right\rangle }^{\left( 2\right) }h_{\left( 1\right)
}^{\prime }\otimes z^{\left( 1\right) }\#z_{\left\langle -1\right\rangle
}^{\left( 2\right) }h_{\left( 1\right) }^{\prime \prime }\right) \upsilon
_{A}\left( x\#h\otimes \left( y_{\left\langle 0\right\rangle }^{\left(
2\right) }\#h_{\left( 2\right) }^{\prime }\right) \left( z_{\left\langle
0\right\rangle }^{\left( 2\right) }\#h_{\left( 2\right) }^{\prime \prime
}\right) \right).  \notag
\end{eqnarray}%
By (\ref{form: multi smash xi}) and the right $H$-linearity of $\upsilon
_{A} $, we have
\begin{equation}
\upsilon _{A}\left( x\#h\otimes \left( y_{\left\langle 0\right\rangle
}^{\left( 2\right) }\#h_{\left( 2\right) }^{\prime }\right) \left(
z_{\left\langle 0\right\rangle }^{\left( 2\right) }\#h_{\left( 2\right)
}^{\prime \prime }\right) \right) =\upsilon _{A}\left( x\#h\otimes
y_{\left\langle 0\right\rangle }^{\left( 2\right) }\left( h_{\left( 2\right)
}^{\prime }z_{\left\langle 0\right\rangle }^{\left( 2\right) }\right)
\#1_{H}\right) \varepsilon_H (h_{\left( 2\right) }^{\prime \prime }) .
\notag
\end{equation}%
Thus, from the definition of $\Omega ^{\prime }$, expression (\ref{eq: lhs})
is equal to
\begin{eqnarray*}
&&\upsilon \left( y^{\left( 1\right) }\otimes y_{\left\langle
-1\right\rangle }^{\left( 2\right) }h_{\left( 1\right) }^{\prime }z^{\left(
1\right) }\right) \varepsilon _{H}\left( z_{\left\langle -1\right\rangle
}^{\left( 2\right) }h_{\left( 1\right) }^{\prime \prime }\right) \upsilon
\left( x\otimes h\left( y_{\left\langle 0\right\rangle }^{\left( 2\right)
}\left( h_{\left( 2\right) }^{\prime }z_{\left\langle 0\right\rangle
}^{\left( 2\right) }\right) \right) \right) \varepsilon _{H}\left( h_{\left(
2\right) }^{\prime \prime }\right) \\
&=&\upsilon \left( y^{\left( 1\right) }\otimes y_{\left\langle
-1\right\rangle }^{\left( 2\right) }h_{\left( 1\right) }^{\prime }z^{\left(
1\right) }\right) \upsilon \left( x\otimes h\left( y_{\left\langle
0\right\rangle }^{\left( 2\right) }\left( h_{\left( 2\right) }^{\prime
}z^{\left( 2\right) }\right) \right) \right) \varepsilon _{H}\left(
h^{\prime \prime }\right).
\end{eqnarray*}%
Then we use left $H$-linearity to obtain that
\begin{eqnarray*}
&&\upsilon \left( y^{\left( 1\right) }\otimes y_{\left\langle
-1\right\rangle }^{\left( 2\right) }h_{\left( 1\right) }^{\prime }z^{\left(
1\right) }\right) \upsilon \left( x\otimes h\left( y_{\left\langle
0\right\rangle }^{\left( 2\right) }\left( h_{\left( 2\right) }^{\prime
}z^{\left( 2\right) }\right) \right) \right) \\
&=&\upsilon \left( x_{\left\langle -2\right\rangle }h_{\left( 1\right)
}y^{\left( 1\right) }\otimes x_{\left\langle -1\right\rangle }h_{\left(
2\right) }y_{\left\langle -1\right\rangle }^{\left( 2\right) }S\left(
h_{\left( 4\right) }\right) \left( h_{\left( 5\right) }h^{\prime }z\right)
^{\left( 1\right) }\right) \upsilon \left( x_{\left\langle 0\right\rangle
}\otimes \left( h_{\left( 3\right) }y_{\left\langle 0\right\rangle }^{\left(
2\right) }\right) \left( \left( h_{\left( 5\right) }h^{\prime }z\right)
^{\left( 2\right) }\right) \right) \\
&=&\upsilon \left( x_{\left\langle -2\right\rangle }\left( h_{\left(
1\right) }y\right) ^{\left( 1\right) }\otimes x_{\left\langle
-1\right\rangle }\left( h_{\left( 1\right) }y\right) _{\left\langle
-1\right\rangle }^{\left( 2\right) }\left( h_{\left( 2\right) }h^{\prime
}z\right) ^{\left( 1\right) }\right) \upsilon \left( x_{\left\langle
0\right\rangle }\otimes \left( h_{\left( 1\right) }y\right) _{\left\langle
0\right\rangle }^{\left( 2\right) }\left( h_{\left( 2\right) }h^{\prime
}z\right) ^{\left( 2\right) }\right)
\end{eqnarray*}%
and thus we have that (\ref{eq: lhs}) equals
\begin{equation*}
\left[ \left( \varepsilon _{R}\otimes \upsilon \right) \ast \upsilon \left(
R\otimes m_{R}\right) \right] \left( x\otimes h_{\left( 1\right) }y\otimes
h_{\left( 2\right) }h^{\prime }z\right) \varepsilon _{H}\left( h^{\prime
\prime }\right)
\end{equation*}%
as claimed. Now we tackle the right hand side of $(\ref{form: cocycle1 Hopf}%
) $ for $\gamma =\upsilon _{A}$.
\begin{eqnarray*}
&& (\upsilon _{A}\otimes \varepsilon _{A}) \ast \upsilon _{A}(m_{A}\otimes
A)(x\#h\otimes y\#h^{\prime }\otimes z\#h^{\prime \prime }) \\
&=&\upsilon _{A}\left( \left( x\#h\right) _{\left( 1\right) }\otimes \left(
y\#h^{\prime }\right) _{\left( 1\right) }\right) \upsilon _{A}\left[ \left(
x\#h\right) _{\left( 2\right) }\left( y\#h^{\prime }\right) _{\left(
2\right) }\otimes z\#h^{\prime \prime }\right] \\
&\overset{(\ref{eq: DeltaA})}{=}&\upsilon _{A}\left( x^{\left( 1\right)
}\#x_{\left\langle -1\right\rangle }^{\left( 2\right) }h_{\left( 1\right)
}\otimes y^{\left( 1\right) }\#y_{\left\langle -1\right\rangle }^{\left(
2\right) }h_{\left( 1\right) }^{\prime }\right) \upsilon _{A}\left[ \left(
x_{\left\langle 0\right\rangle }^{\left( 2\right) }\#h_{\left( 2\right)
}\right) \left( y_{\left\langle 0\right\rangle }^{\left( 2\right)
}\#h_{\left( 2\right) }^{\prime }\right) \otimes z\#h^{\prime \prime }\right]
\\
&=&\upsilon \left( x^{\left( 1\right) }\otimes x_{\left\langle
-1\right\rangle }^{\left( 2\right) }h_{\left( 1\right) }y^{\left( 1\right)
}\right) \varepsilon _{H}\left( y_{\left\langle -1\right\rangle }^{\left(
2\right) }h_{\left( 1\right) }^{\prime }\right) \upsilon _{A}\left[ \left(
x_{\left\langle 0\right\rangle }^{\left( 2\right) }\#h_{\left( 2\right)
}\right) \left( y_{\left\langle 0\right\rangle }^{\left( 2\right)
}\#h_{\left( 2\right) }^{\prime }\right) \otimes z\#h^{\prime \prime }\right]
\\
&=&\upsilon \left( x^{\left( 1\right) }\otimes x_{\left\langle
-1\right\rangle }^{\left( 2\right) }h_{\left( 1\right) }y^{\left( 1\right)
}\right) \upsilon _{A}\left[ \left( x_{\left\langle 0\right\rangle }^{\left(
2\right) }\#h_{\left( 2\right) }\right) \left( y^{\left( 2\right)
}\#h^{\prime }\right) \otimes z\#h^{\prime \prime }\right] .
\end{eqnarray*}%
But by (\ref{form: multi smash xi}),
\begin{eqnarray*}
\left( x_{\left\langle 0\right\rangle }^{\left( 2\right) }\#h_{\left(
2\right) }\right) \left( y^{\left( 2\right) }\#h^{\prime }\right) &=&\left(
R\otimes m_{H}\right) \left[ \left( m_{R}\otimes \xi \right) \Delta
_{R\otimes R}\otimes H\right] \left( x_{\left\langle 0\right\rangle
}^{\left( 2\right) }\otimes h_{\left( 2\right) }y^{\left( 2\right) }\otimes
h_{\left( 3\right) }h^{\prime }\right) \\
&=&\left( x_{\left\langle 0\right\rangle }^{\left( 2\right) }\right)
^{\left( 1\right) }\left( \left( x_{\left\langle 0\right\rangle }^{\left(
2\right) }\right) _{\left\langle -1\right\rangle }^{\left( 2\right)
}h_{\left( 2\right) }y^{\left( 2\right) }\right) \#\xi \left[ \left(
x_{\left\langle 0\right\rangle }^{\left( 2\right) }\right) _{\left\langle
0\right\rangle }^{\left( 2\right) }\otimes h_{\left( 3\right) }y^{\left(
3\right) }\right] h_{\left( 4\right) }h^{\prime },
\end{eqnarray*}%
and
\begin{eqnarray*}
&&\upsilon _{A}\left[ \left( x_{\left\langle 0\right\rangle }^{\left(
2\right) }\right) ^{\left( 1\right) }\left( \left( x_{\left\langle
0\right\rangle }^{\left( 2\right) }\right) _{\left\langle -1\right\rangle
}^{\left( 2\right) }h_{\left( 2\right) }y^{\left( 2\right) }\right) \#\xi %
\left[ \left( x_{\left\langle 0\right\rangle }^{\left( 2\right) }\right)
_{\left\langle 0\right\rangle }^{\left( 2\right) }\otimes h_{\left( 3\right)
}y^{\left( 3\right) }\right] h_{\left( 4\right) }h^{\prime }\otimes
z\#h^{\prime \prime }\right] \\
&=&\upsilon \left[ \left( x_{\left\langle 0\right\rangle }^{\left( 2\right)
}\right) ^{\left( 1\right) }\left( \left( x_{\left\langle 0\right\rangle
}^{\left( 2\right) }\right) _{\left\langle -1\right\rangle }^{\left(
2\right) }h_{\left( 2\right) }y^{\left( 2\right) }\right) \otimes \xi \left[
\left( x_{\left\langle 0\right\rangle }^{\left( 2\right) }\right)
_{\left\langle 0\right\rangle }^{\left( 2\right) }\otimes h_{\left( 3\right)
}y^{\left( 3\right) }\right] h_{\left( 4\right) }h^{\prime }z\right]
\varepsilon _{H}\left( h^{\prime \prime }\right)
\end{eqnarray*}%
Thus $\upsilon _{A}(A\otimes A\otimes \varepsilon _{A})\ast \upsilon
_{A}(m_{A}\otimes A)(x\#h\otimes y\#h^{\prime }\otimes z\#h^{\prime \prime
}) $ is exactly
\begin{equation*}
\left[ \left( \upsilon \otimes \varepsilon _{R}\right) \ast \left\{ \upsilon
\left( m_{R}\otimes R\right) \Phi \left( \xi \right) \right\} \right] \left(
x\otimes h_{\left( 1\right) }y\otimes h_{\left( 2\right) }h^{\prime
}z\right) \varepsilon _{H}\left( h^{\prime \prime }\right) .
\end{equation*}%
This proves the theorem.
\end{proof}

Recall from Lemma \ref{re: splitting} that if $(A,H,\pi,\sigma)$ is a
splitting datum with associated pre-bialgebra with cocycle $(R, \xi)$, and
if $\gamma \in Z^2_H(A,K)$, then $(A^\gamma, H,\pi,\sigma)$ is also a
splitting datum and has associated pre-bialgebra with cocycle $(R, \eta)$
for some $\eta$. The next theorem describes this relationship more precisely.

\begin{theorem}
\label{teo: smash gamma}Let $(R,\xi )$ be a pre-bialgebra with cocycle, and
let $\gamma \in Z_{H}^{2}(R\#_{\xi }H,K)$. Define $\xi _{\gamma
_{R}}:R^{\gamma _{R}}\otimes R^{\gamma _{R}}\rightarrow H$ by
\begin{equation*}
\xi _{\gamma _{R}}=u_{H}\gamma _{R}\ast \xi \ast \Psi (\gamma
_{R}^{-1})=u_{H}\gamma _{R}\ast \xi \ast \left( H\otimes \gamma
_{R}^{-1}\right) \rho _{R\otimes R}.
\end{equation*}
Let $(A:=R\#_{\xi }H,H,\pi ,\sigma )$ be the splitting datum of \ref{sec:
splittingdatum} so that by \ref{sec: correspondence}, the associated
pre-bialgebra with cocycle is $(R \otimes K, \xi (\theta \otimes \theta))$,
where $\theta:R\otimes K\rightarrow R$ is the usual isomorphism. Then $%
\left( A^{\gamma }=(R\#_{\xi }H\right) ^{\gamma },H,\pi ,\sigma )$ is also a
splitting datum whose associated pre-bialgebra with cocycle is $(R\otimes
K,\xi _{\gamma _{R}}(\theta\otimes \theta))$. Furthermore $(R^{\gamma },\xi
_{\gamma _{R}})$ is a pre-bialgebra with cocycle isomorphic to $(R\otimes
K,\xi _{\gamma _{R}})$ via $\theta$ and
\begin{equation*}
A^{\gamma }=\left( R\#_{\xi }H\right) ^{\gamma }=R^{\gamma _{R}}\#_{\xi
_{\gamma _{R}}}H.
\end{equation*}
\end{theorem}

\begin{proof}
By Lemma \ref{lem: Psi}, $\Psi (\gamma _{R}^{-1})=\left( H\otimes \gamma
_{R}^{-1}\right) \rho _{R\otimes R} $ is convolution invertible with inverse
$\Psi (\gamma _{R})$. For $A:= R\#_{\xi }H $ with associated pre-bialgebra
with cocycle $(R \otimes K, \xi)$, let $\left( Q = R \otimes K,\zeta \right)
$ denote the pre-bialgebra with cocycle associated to $(A^{\gamma },H,\pi
,\sigma ) $. For $x\otimes 1,y \otimes 1 \in Q$, since $\tau = R \otimes
\varepsilon_H$ from Section \ref{sec: correspondence}, multiplication in $Q$
is given by
\begin{eqnarray*}
\lefteqn{ \left( x\otimes 1_{K}\right) \cdot _{Q}\left( y\otimes 1_{K}\right)%
} \\
&=&\tau \left[ \left( x\#1_{H}\right) \cdot _{A^{\gamma }}\left(
y\#1_{H}\right) \right] \\
& =&\left( R\otimes \varepsilon _{H}\right) \left[ \left( x\#1_{H}\right)
\cdot _{A^{\gamma }}\left( y\#1_{H}\right) \right] \\
&\overset{(\ref{form: varepsilonHmA})}{=}& \gamma _{R}\left( x^{(1)}\otimes
x_{\langle -1\rangle }^{(2)}x_{\langle -3\rangle }^{(3)}y^{(1)}\right)
(x_{\langle 0\rangle }^{(2)})\cdot _{R}(x_{\langle -1\rangle
}^{(3)}y^{(2)})\otimes \varepsilon _{H}(y_{\langle -1\rangle }^{(3)})\gamma
_{R}^{-1}\left( x_{\langle 0\rangle }^{(3)}\otimes y_{\langle 0\rangle
}^{(3)}\right) \\
&=&\gamma _{R}\left( x^{(1)}\otimes x_{\langle -1\rangle }^{(2)}x_{\langle
-2\rangle }^{(3)}y^{(1)}\right) x_{\langle 0\rangle }^{(2)}\cdot _{R}\left(
x_{\langle -1\rangle }^{(3)}y^{(2)}\right) \gamma _{R}^{-1}\left( x_{\langle
0\rangle }^{(3)}\otimes y^{(3)}\right) \otimes 1_{K} \\
&=&\left[ \left( \gamma _{R}\otimes m_{R}\otimes \gamma _{R}^{-1}\right)
\Delta _{R\otimes R}^{2}\left( x\otimes y\right) \right] \otimes 1_{K} \\
&=&m_{R^{\gamma _{R}}}\left( x\otimes y\right) \otimes 1_{K}.
\end{eqnarray*}

Furthermore, we have
\begin{eqnarray*}
\lefteqn{ \zeta \left( x\otimes 1_{K}\otimes y\otimes 1_{K}\right)} \\
&=&\pi \left[ \left( x\#1_{H}\right) \cdot _{A^{\gamma }}\left(
y\#1_{H}\right) \right] \\
&=&\left( \varepsilon _{R}\otimes H\right) \left[ \left( x\#1_{H}\right)
\cdot _{A^{\gamma }}\left( y\#1_{H}\right) \right] \\
&\overset{(\ref{form: varepsilonRmA})}{=}& \gamma _{R}\left( x^{(1)}\otimes
x_{\langle -1\rangle }^{(2)}x_{\langle -3\rangle }^{(3)}y^{(1)}\right) \xi
(x_{(0)}^{(2)}\otimes x_{\langle -2\rangle }^{(3)}y^{(2)})x_{\langle
-1\rangle }^{(3)}y_{\langle -1\rangle }^{(3)}\gamma _{R}^{-1}\left(
x_{\langle 0\rangle }^{(3)}\otimes y_{\langle 0\rangle }^{(3)}\right) \\
&=&\gamma _{R}\left( x^{(1)}\otimes x_{\langle -1\rangle }^{(2)}x_{\langle
-2\rangle }^{(3)}y^{(1)}\right) \xi \left( x_{\langle 0\rangle
}^{(2)}\otimes x_{\langle -1\rangle }^{(3)}y^{(2)}\right) \left( H\otimes
\gamma _{R}^{-1}\right) \rho _{R\otimes R}\left( x_{\langle 0\rangle
}^{(3)}\otimes y^{(3)}\right) \\
&=&\left[ u_{H}\gamma _{R}\ast \xi \ast \left( H\otimes \gamma
_{R}^{-1}\right) \rho _{R\otimes R}\right] \left( x\otimes y\right) \\
&=&\xi _{\gamma _{R}}\left( x\otimes y\right).
\end{eqnarray*}

Thus $\theta$ , the isomorphism of coalgebras in $_H^H\mathcal{YD}$ from
Section \ref{sec: correspondence}, induces the structure of a pre-bialgebra
with cocycle on the image $\theta(Q)$ and the image is exactly $R^{\gamma_R}$%
. Moreover $(R^{\gamma_R}, \xi_{\gamma_R})$ is indeed a pre-bialgebra with
cocycle as claimed, and $\zeta = \xi_{\gamma_R}$.

As in Section \ref{sec: A co H}, we can consider the bialgebra isomorphism
\begin{equation*}
\omega ^{-1}:A^{\gamma }\rightarrow Q\#_{\zeta }H\text{, }\omega
^{-1}(a)=\tau \left( a_{\left( 1\right) }\right) \otimes \pi \left(
a_{\left( 2\right) }\right) .
\end{equation*}
We have
\begin{eqnarray*}
\omega ^{-1}\left( r\otimes h\right) &=&\tau \left( r^{(1)}\#r_{\langle
-1\rangle }^{(2)}h_{(1)}\right) \otimes \pi \left( r_{\langle 0\rangle
}^{(2)}\#h_{\left( 2\right) }\right) \\
&=&\left( r^{(1)}\otimes \varepsilon _{H}\left( r_{\langle -1\rangle
}^{(2)}h_{(1)}\right) \right) \otimes \varepsilon _{R}\left( r_{\langle
0\rangle }^{(2)}\right) h_{\left( 2\right) } \\
&=&r\otimes 1_{K}\otimes h.
\end{eqnarray*}%
Since $\theta:\left( Q,\zeta \right) \rightarrow (R^{\gamma },\xi _{\gamma
_{R}})$ is an isomorphism of pre-bialgebras with cocycle (see Section \ref%
{sec: pre-bialg}) we get that
\begin{equation*}
\theta \otimes H:Q\#_{\zeta }H\rightarrow R^{\gamma _{R}}\#_{\xi _{\gamma
_{R}}}H
\end{equation*}%
is a bialgebra isomorphism. Therefore $\left( \theta \otimes H\right) \omega
^{-1}:A^{\gamma }\rightarrow R^{\gamma _{R}}\#_{\xi _{\gamma _{R}}}H$ is a
bialgebra isomorphism too. By the foregoing $\mathrm{Id}_{R\otimes H}=\left(
\theta \otimes H\right) \omega ^{-1}.$
\end{proof}

\begin{corollary}
Let $(A,H,\pi ,\sigma )$ be a splitting datum, let $(R,\xi )$ be the
associated pre-bialgebra with cocycle, and let $\gamma \in Z_{H}^{2}(A,K)$.
Then $(R^{\gamma _{R}},\xi _{\gamma _{R}})$ is exactly the pre-bialgebra
with cocycle associated to the splitting datum $(A^{\gamma },H,\pi ,\sigma )$%
.
\end{corollary}

\begin{proof}
As in Section \ref{sec: A co H}, we can consider the bialgebra isomorphisms%
\begin{equation*}
\omega :R\#_{\xi }H\rightarrow A\text{, }\omega (r\otimes h)=r\cdot
_{A}\sigma (h).
\end{equation*}%
This gives a bialgebra isomorphism
\begin{equation*}
\omega ^{\gamma }:\left( R\#_{\xi }H\right) ^{\gamma \left( \omega \otimes
\omega \right) }\rightarrow A^{\gamma }\text{, }\omega ^{\gamma }(r\otimes
h)=\omega (r\otimes h).
\end{equation*}%
Let $\alpha :=\gamma \left( \omega \otimes \omega \right) .$ By Theorem \ref%
{teo: smash gamma}, we have
\begin{equation*}
\left( R\#_{\xi }H\right) ^{\alpha }=R^{\alpha _{R}}\#_{\xi _{\alpha _{R}}}H.
\end{equation*}%
We have
\begin{equation*}
\alpha _{R}\left( r\otimes s\right) =\alpha \left( r\otimes 1_{H}\otimes
s\otimes 1_{H}\right) =\gamma \left( \omega \otimes \omega \right) \left(
r\otimes 1_{H}\otimes s\otimes 1_{H}\right) =\gamma \left( r\otimes s\right)
=\gamma _{R}\left( r\otimes s\right)
\end{equation*}%
so that $\alpha _{R}=\gamma _{R}$ whence
\begin{equation*}
\left( R\#_{\xi }H\right) ^{\alpha }=R^{\gamma _{R}}\#_{\xi _{\gamma _{R}}}H.
\end{equation*}%
Denote by $\left( Q,\zeta \right) $ the pre-bialgebra associated to $%
(A^{\gamma },H,\pi ,\sigma ).$ We have%
\begin{equation*}
Q=\left( A^{\gamma }\right) ^{coH}=\left( A\right) ^{coH}=R
\end{equation*}%
so that $Q\#_{\zeta }H=R\otimes H$ as a vector space. The isomorphism
corresponding to the splitting datum $(A^{\gamma },H,\pi ,\sigma )$ is given
by%
\begin{equation*}
\omega ^{\prime }:Q\#_{\zeta }H\rightarrow A^{\gamma }\text{, }\omega
^{\prime }(r\otimes h)=r\cdot _{A^{\gamma }}\sigma (h).
\end{equation*}%
We have%
\begin{eqnarray*}
r\cdot _{A^{\gamma }}\sigma (h) &=&\gamma \left( r_{\left( 1\right) }\otimes
\sigma (h)_{\left( 1\right) }\right) r_{\left( 2\right) }\cdot _{A}\sigma
(h)_{\left( 2\right) }\gamma ^{-1}\left( r_{\left( 3\right) }\otimes \sigma
(h)_{\left( 3\right) }\right) \\
&=&\gamma \left( r_{\left( 1\right) }\otimes \sigma (h_{\left( 1\right)
})\right) r_{\left( 2\right) }\cdot _{A}\sigma (h_{\left( 2\right) })\gamma
^{-1}\left( r_{\left( 3\right) }\otimes \sigma (h_{\left( 3\right) })\right)
\\
&=&\gamma \left( r_{\left( 1\right) }\sigma (h_{\left( 1\right) })\otimes
1_{A}\right) r_{\left( 2\right) }\cdot _{A}\sigma (h_{\left( 2\right)
})\gamma ^{-1}\left( r_{\left( 3\right) }\sigma (h_{\left( 3\right)
})\otimes 1_{A}\right) \\
&=&r\cdot _{A}\sigma (h)
\end{eqnarray*}%
so that $\omega ^{\prime }=\omega ^{\gamma }.$ Hence we get that
\begin{equation*}
\mathrm{Id}_{R\otimes H}=\left( \omega ^{\prime }\right) ^{-1}\omega
^{\gamma }:R^{\gamma _{R}}\#_{\xi _{\gamma _{R}}}H\rightarrow Q\#_{\zeta }H
\end{equation*}%
is a bialgebra isomorphism. Now $( R^{\gamma _{R}}\#_{\xi _{\gamma
_{R}}}H,H,\pi ^{\prime },\sigma ^{\prime } ) $ and $\left( Q\#_{\zeta
}H,H,\pi ^{\prime },\sigma ^{\prime }\right) $ are both splitting data where
$\pi ^{\prime }:R\otimes H\rightarrow H,\pi ^{\prime }\left( r\otimes
h\right) =\varepsilon _{A}\left( r\right) h$ and $\sigma ^{\prime
}:H\rightarrow R\otimes H,\sigma ^{\prime }\left( h\right) =1_{R}\otimes h.$
Therefore, in view of \cite[Proposition 1.15]{AM-Small2}, one gets that $%
(R^{\gamma _{R}},\xi _{\gamma _{R}})=\left( Q,\zeta \right) $ as
pre-bialgebras with cocycle.
\end{proof}

We consider when the conditions which ensure associativity of $R$ also hold
for a cocycle twist $R^\upsilon$.

\begin{corollary}
Let $A = R \#_\xi H$, $\gamma \in Z^2_H(A,K)$, and $A^\gamma = R^{\upsilon}
\#_{\xi_\upsilon} H$ as in Theorem \ref{teo: smash gamma} where $\upsilon :=
\Omega(\gamma) = \gamma_R$ and $\xi_\upsilon := \xi_{ \gamma_R}$. By Theorem %
\ref{te: Rassoc}, we have

\textrm{{(a)} $\xi \left( z\right) t=\varepsilon \left( z\right) t,$ for
every $z\in R \otimes R,t\in R$ if and only if $\Phi \left( \xi \right) =%
\mathrm{Id} _{R^{\otimes 3}}$,\newline
and, }

\textrm{{(b)} $\xi_\upsilon \left( z\right) t=\varepsilon \left( z\right) t,
$ for every $z\in R \otimes R,t\in R$ if and only if $\Phi \left(
\xi_\upsilon \right) =\mathrm{Id} _{R^{\otimes 3}}$. }

\textrm{If $\Phi(\Psi(\upsilon)) = \Phi(u_H\upsilon)$ then {(a)} and {(b)}
are equivalent. Conversely, if {(a)} and {(b)} both hold, then $%
\Phi(\Psi(\upsilon)) = \Phi(u_H\upsilon)$. }
\end{corollary}

\begin{proof}
Suppose first that $\Phi(\Psi(\upsilon)) = \Phi(u_H\upsilon)$. Since by
Theorem \ref{teo: smash gamma}, $\Phi(\xi_\upsilon) = \Phi(u_H \upsilon *
\xi * \Psi(\upsilon^{-1}))$, and by Lemma \ref{lem: Phi}, $\Phi$ is an
algebra map, then clearly $\Phi(\xi_\upsilon)$ is the identity if and only
if $\Phi(\xi )$ is. Conversely if $\Phi \left( \xi_\upsilon \right) =\mathrm{%
Id} _{R^{\otimes 3}} = \Phi \left( \xi \right) $, then by Theorem \ref{teo:
smash gamma}, $\Phi(u_H \upsilon * \Psi(\upsilon^{-1})) = \mathrm{Id}
_{R^{\otimes 3}}$.
\end{proof}

\section{Cocycle twists for Radford biproducts of quantum planes}

\label{sec: qlps}

\subsection{Construction of the cocycle}

Let $\Gamma$ be a finite abelian group, let $H = K[\Gamma]$ and let $V =
Kx_1 \oplus Kx_2 \in {_\Gamma^\Gamma \mathcal{YD}}$ be a quantum plane with $%
x_i \in V^{\chi_i}_{g_i}$ as in Definition \ref{de: qls}. Let $A $ be the
Radford biproduct $\mathcal{B}(V) \# H$. Suppose as well that $g_1g_2 \neq 1$%
, $\chi_1\chi_2 = \varepsilon$, $g_i^r \neq 1$ and $\chi_i^r = \varepsilon$
so that it makes sense for the scalars $a_i$ and $a$ to be nonzero.

\vspace{1mm}

Let $\chi := \chi_1 = \chi_2^{-1}$. Suppose that $\chi(g_1)=q$ is a
primitive $r$th root of unity. Then $\chi_1(g_1) = \chi_1(g_2) = q $ and $%
\chi_2(g_1) = \chi_2(g_2) = q^{-1}$.

\vspace{1mm}

Although it is known that liftings of the coradically graded Hopf algebra $A$
are isomorphic to cocycle twists of $A$, the explicit description of the
cocycle in the most general setting has not been given. In this section we
will describe this cocycle.

\begin{lemma}
\label{le: qlpfirstlemma} Let $\gamma \in \mathrm{{Hom}(A \otimes A , K)}$
be $H$-bilinear and $H$-balanced. If $q^{i+k} \neq q^{j+l}$ then $%
\gamma(x_1^ix_2^j \otimes x_1^kx_2^l) =0 $.
\end{lemma}

\begin{proof}
Suppose $\gamma(x_1^ix_2^j \otimes x_1^kx_2^l)\neq 0$. Since $\gamma$ is $H $%
-balanced, $\gamma(x_1^ix_2^jg_1 \otimes x_1^kx_2^l) = \gamma(x_1^ix_2^j
\otimes g_1x_1^kx_2^l)$ so that from the $H$-bilinearity of $\gamma$ then $%
\chi^{j-i}(g_1) = \chi^{k-l}(g_1)$. Thus $q^{j-i} = q^{k-l}$ and the result
follows.
\end{proof}

\begin{corollary}
\label{co: qlpcocycle} Let $\gamma \in Z^2_H(A ,K)$. If $q^{i+k} \neq
q^{j+l} $, then $\gamma^{\pm1}(x_1^ix_2^j \otimes x_1^kx_2^l)=0.$
\end{corollary}

\begin{proof}
By Lemma \ref{lem: balanced}, $\gamma$ is $H$-balanced, and $\gamma^{-1}$ is
$H$-bilinear and $H$-balanced also.
\end{proof}

\vspace{1mm} The next propositions will require the $q$-analogue of the
Chu-Vandermonde formula \cite[Proposition IV.2.3]{Kassel}
\begin{equation}  \label{eq: kassel}
\sum_{k=0}^r \binom{a}{k}_q \binom{b}{r-k}_q q^{(a-k)(r-k)} = \binom{a+b}{r}%
_q,
\end{equation}
as well as the fact that when $q$ is a primitive $r^{th}$ root of unity and $%
k \leq r$
\begin{equation}  \label{eq: gm equality}
\binom{r+k}{k}_q = 1.
\end{equation}

Also we will need the fact, which follows directly from the $q$-binomial
theorem, that
\begin{equation}  \label{eq: qminus1}
\binom{n}{k}_{q^{-1}}q^{k(n-k)} = \binom{n}{k}_q.
\end{equation}
If $n,i$ or $n-i$ is negative, we set $\binom{n}{i}_q =0$.

\vspace{1mm}

\begin{proposition}
\label{pr: gamma(ai)} Let $A = R \# H$ as above with $R = \mathcal{B}(V)$, $%
H = K[\Gamma]$. Define $H$-bilinear maps $\gamma_i$, $i=1,2$, from $A
\otimes A$ to $K$ as follows: $\gamma_i= \varepsilon$ on all $x_1^lx_2^m
\otimes x_1^n x_2^t$ except that
\begin{equation}  \label{eq: gammai}
\gamma_i(x_i^m \otimes x_i^{r-m}) = a_i,
\end{equation}
and $\gamma_i$ is then extended to $A \otimes A$ by $H$-bilinearity.

The maps $\gamma_1, \gamma_2$ lie in $Z^2_H(A,K)$ and furthermore these
cocycles commute.
\end{proposition}

\begin{proof}
We show first that $\gamma_1 \in Z^2_H(A,K)$. Note that by the definition, $%
\gamma_i$ is $H$-balanced. By the definition of $\gamma_i$, condition $(\ref%
{form: cocycle3})$ holds and we check condition (\ref{form: cocycle1 Hopf})
for the triple $x_1^ix_2^j, x_1^kx_2^l, x_1^tx_2^s $. It is clear from the
definition of $\gamma_1 $ that both sides of (\ref{form: cocycle1 Hopf}) are
$0$ unless $j=l=s=0$ so that the triple to check is $x_1^i, x_1^k, x_1^t $.
By Lemma \ref{le: qlpfirstlemma}, both sides are $0$ unless $i+k+t = r$ or $%
i+k+t = 2r$. Suppose first that $i+k+t =r$. Then the left hand side equals $%
\gamma_1(g_1^k \otimes g_1^t)\gamma_1(x_1^i \otimes x_1^kx_1^t) = a_1$ and
similarly the right hand side is $\gamma_1(x_1^ix_1^k \otimes x_1^t) = a_1$.

Now let $i+k+t = 2r$. Then the left hand side of (\ref{form: cocycle1 Hopf})
is
\begin{eqnarray*}
&& \sum_{m \geq 0} \binom{k}{m}_{q^{-1}}\binom{t}{r-m}_{q} \gamma_1 (x_1^m
\otimes x_1^{r-m})\gamma_1(x_1^i \otimes x_1^{k-m}x_1^{t-(r-m)}) \\
&& \overset{(\ref{eq: qminus1})}{=} (\sum_{m \geq 0} q^{-m(k-m)} \binom{k}{m}%
_q \binom{t}{r-m}_{q} ) a_1^2 \\
&& \overset{(\ref{eq: kassel})}{=} \binom{k+t}{r}_q a_1^2 \overset{(\ref{eq:
gm equality})}{=} a_1^2.
\end{eqnarray*}

Similarly the right hand side of (\ref{form: cocycle1 Hopf}) is
\begin{equation*}
\sum_{n \geq 0} \binom{i}{n}_{q^{-1}}\binom{k}{r-n}_{q} \gamma_1 (x_1^n
\otimes x_1^{r-n})\gamma_1(x_1^{i-n}x_1^{k-(r-n)} \otimes x_1^{t }) = \binom{%
i+k}{r}_q a_1^2 = a_1^2.
\end{equation*}

The proof that $\gamma_2$ is a cocycle is analogous.

We show next that $\gamma_1 $ and $\gamma_2$ commute by applying $\gamma_1
\ast \gamma_2$ and $\gamma_2 \ast \gamma_1$ to $x_1^ix_2^j \otimes
x_1^kx_2^l $. 
Then
\begin{equation*}
\gamma_1 \ast \gamma_2(x_1^ix_2^j \otimes x_1^kx_2^l) = \gamma_1(x_1^ig_2^j
\otimes x_1^k)\gamma_2(x_2^j \otimes x_2^l) =
q^{-ij}\delta_{i+k,r}\delta_{j+l,r}a_1a_2,
\end{equation*}
while
\begin{equation*}
\gamma_2 \ast \gamma_1(x_1^ix_2^j \otimes x_1^kx_2^l)= \gamma_2( x_2^j
\otimes g_1^kx_2^l)\gamma_1(x_1^i \otimes x_1^k) =
q^{-kl}\delta_{i+k,r}\delta_{j+l,r}a_1a_2.
\end{equation*}
Since for $i+k = j+l =r$ then $q^{-ij} = q^{-(r-k)(r-l)} = q^{-kl}$, these
expressions are equal.
\end{proof}

Note that it is straightforward to check that $\gamma_i^{-1}$ is the $H$%
-bilinear map defined exactly as $\gamma_i$ is but with $a_i$ replaced by $%
-a_i$. Also note that the multiplication $m_i: A^{\gamma_i} \otimes
A^{\gamma_i} \rightarrow A^{\gamma_i}$ is the same as the multiplication on $%
A$ except that for $0 < m <r$,
\begin{equation*}
m_i(x_i^m \otimes x_i^{r-m}) = \gamma_i(x_i^m \otimes x_i^{r-m}) + x_i^r +
g_i^r \gamma_i^{-1}(x_i^m \otimes x_i^{r-m}) = a_i(1-g_i^r).
\end{equation*}

\begin{corollary}
\label{co: gammai cocycle}For $i \neq j$, $i,j = 1,2$, $\gamma_i \in
Z^2_H(A^{\gamma_j},K)$.
\end{corollary}

\begin{proof}
Basically the same proof as that of Proposition \ref{pr: gamma(ai)} shows
that $\gamma_i \in Z^2_H(A^{\gamma_j} ,K)$, $i \neq j$. For example, to show
that $\gamma_2 \in Z^2_H(A^{\gamma_1},K)$, we test the triple $x_1^ix_2^j,
x_1^kx_2^l, x_1^tx_2^s $ and find that the left hand side of (\ref{form:
cocycle1 Hopf}) is
\begin{eqnarray*}
&& \sum_m \binom{l}{m}_{q} \binom{s}{r-m}_{q^{-1}} \gamma_2(g_1^k x_2^m
\otimes g_1^tx^{r-m}_2)\gamma_2(x_1^ix_2^j \otimes x_1^k
x_2^{l-m}x_1^tx_2^{s-r+m}) \\
&& = \sum_m \binom{l}{m}_{q} \binom{s}{r-m}_{q^{-1}} q^{(l-m)t} \gamma_2(
x_2^m \otimes g_1^kx^{r-m}_2)\gamma_2(x_1^{i} x_2^{j } \otimes x_1^{k+t}
x_2^{l-m} x_2^{s -r+m}).
\end{eqnarray*}

Clearly this is $0$ unless $i=0$. If $0 < k+t \neq r$, then this expression
is also clearly $0$. If $k+t = r$ then $x_1^{k+t} = a_1(1-g_1^r)$ and since $%
g_1^r$ commutes with $x_2$ and $\gamma_2$ is $H$-bilinear, we have $0$ here
too. Thus the left hand side is $0$ unless $t=k=i=0$ and the right hand side
computation is similar. Thus the computation simplifies to that in
Proposition \ref{pr: gamma(ai)}.
\end{proof}

\begin{corollary}
\label{co: building cocycles} For $i,j = 1,2$ and $i \neq j$, then $\gamma_i
\ast \gamma_j \in Z^2_H(A,K)$.
\end{corollary}

\begin{proof}
  By Proposition \ref{pr: gamma(ai)}, $\gamma _{j}\in
Z_{H}^{2}(A,K)$ and by{\ Corollary \ref{co: gammai cocycle} we
have }$\gamma _{i}\in Z_{H}^{2}(A^{\gamma _{j}},K).$  The
statement then follows from Corollary \ref{co: beattie}.
\end{proof}

\vspace{1mm}

Note that the multiplication $m^\prime$ in $A^{\gamma_1 \ast \gamma_2}$ is
the same as that of $A$ except that for $0 < l,m <r $
\begin{equation*}
m^\prime(x_1^lx_2^m \otimes x_1^{r-l}x_2^{r-m}) =
q^{-lm}a_1a_2(1-g_1^r)(1-g_2^r).
\end{equation*}

Now we consider a cocycle which twists the multiplication of $x_1$ and $x_2$.

\begin{proposition}
\label{gamma a} Let $A = R \# H$ as above with $R = \mathcal{B}(V)$, $H =
K[\Gamma]$. Define the $H$-bilinear map $\gamma_a$ from $A \otimes A$ to $K$
as follows: $\gamma_a= \varepsilon$ on all $x_1^lx_2^m \otimes x_1^n x_2^t$
except that
\begin{equation}
\gamma_a(x_2^m \otimes x_1^m) = (m)!_q a^m,
\end{equation}
and $\gamma_a$ is then extended to all of $A \otimes A$ by $H$-bilinearity.
Then $\gamma_a \in Z^2_H(A^{\gamma_1 \ast \gamma_2},K)$.
\end{proposition}

\begin{proof}
Let $\beta$ denote $\gamma_1 \ast \gamma_2 = \gamma_2 \ast \gamma_1$. We
check that $\gamma_a \in Z^2_H(A^\beta,K)$ by applying the left and right
hand sides of equation (\ref{form: cocycle1 Hopf}) to the triple $%
x_1^ix_2^j, x_1^kx_2^l, x_1^tx_2^s$.

The left hand side is equal to:
\begin{eqnarray*}
&& \sum_{m} \binom{l}{m}_q \binom{t}{m}_q \gamma_a(g_1^k x_2^m \otimes
x_1^mg_2^s)\gamma_a(x_1^i x_2^j \otimes x_1^k x_2^{l-m} x_1^{t-m } x_2^s) \\
&& = \delta_{i,0} \sum_{m} \binom{l}{m}_q \binom{t}{m}_q (m)!_q a^m
q^{(l-m)(t-m)}\gamma_a( x_2^j \otimes x_1^k x_1^{t-m }x_2^{l-m} x_2^s).
\end{eqnarray*}
This expression is $0$ unless $s=0$ and $l=m$. If $l-m+s \neq r$ this is
clear, and if $l-m+s = r$, then we have $\gamma_a(x_2^j \otimes x_1^k
x_1^{t-m}(a_2(1-g_2^r)))$ which is $0$ by the $H$-bilinearity of $\gamma_a$.
Thus :
\begin{equation*}
\mathit{lhs} = \delta_{i,0} \delta_{s,0} \delta_{k+t,l+j} \binom{t}{l}_q
(l)!_q (j)!_q a^{j+l}.
\end{equation*}

Similarly the right hand side equals:
\begin{equation*}
\delta_{i,0} \delta_{s,0}\binom{j}{k}_q \gamma_a(x_2^k \otimes
x_1^k)\gamma_a(x_2^{j-k} x_2^l \otimes x_1^t) = \delta_{i,0} \delta_{s,0}
\binom{j}{k}_q \delta_{j-k+l,t} (k)!_q (t)!_q a^{k+t},
\end{equation*}
and it is an easy exercise to see that $\binom{t}{l}_q (l)!_q (j)!_q =
\binom{j}{k}_q (k)!_q (t)!_q$. Thus $\gamma_a \in Z^2_H(A^\beta,K)$.
\end{proof}

We note that the cocycles $\gamma_i$ and $\gamma_a$ do not commute. For
example, consider
\begin{equation*}
\gamma_1 \ast \gamma_a(x_1^{r-1}x_2 \otimes x_1^2) = \binom{2}{1}%
_q\gamma_1(x_1^{r-1}g_2 \otimes x_1)\gamma_a(x_2 \otimes x_1) = q\binom{2}{1}%
_q a_1a,
\end{equation*}
while
\begin{equation*}
\gamma_a \ast \gamma_1(x_1^{r-1}x_2 \otimes x_1^2) = \binom{2}{1}%
_q\gamma_a(g_1^{r-1}x_2 \otimes x_1)\gamma_1(x_1^{r-1} \otimes x_1) = \binom{%
2}{1}_q a_1a.
\end{equation*}
Similar examples show that $\gamma_2$ and $\gamma_a$ do not commute.

\vspace{1mm}

\begin{corollary}
$\gamma_a \in Z^2_H(A,K)$ and $\gamma_a \in Z^2_H(A^{\gamma_i},K)$ for $i =
1,2$.
\end{corollary}

\begin{proof}
Note that the $a_i$ are any scalars. If $a_i = 0$ then $A^{\gamma_i\ast
\gamma_j} = A^{\gamma_j}$ and if $a_1=a_2 =0$ then $A^{\gamma_i\ast
\gamma_j} = A$.
\end{proof}

Note that the cocycles $\gamma_a \in Z^2_H(A,K) $ and $\gamma_i $ (for a
quantum line) were described in \cite[Section 5.3]{Grunenfelder-Mastnak} in
terms of Hochschild cohomology.

\begin{corollary}
 $\alpha:= \gamma_a \ast \gamma_1 \ast \gamma_2 \in Z^2_H(A,K).$

\end{corollary}

\begin{proof}
  By Corollary {\ref{co: building cocycles}, }$\gamma
_{1}\ast \gamma _{2}\in Z_{H}^{2}(A,K)$ and, by Proposition
\ref{gamma a}, $\gamma _{a}\in Z_{H}^{2}(A^{\gamma _{1}\ast \gamma
_{2}},K).$  The statement then follows from  Corollary \ref{co:
beattie}.
\end{proof}

We now describe the cocycle twist of $A^{\alpha }$ of $%
A$. We will need the fact that $\gamma_a^{-1}(x_2 \otimes x_1) =
-a$; this is easy to check.

\begin{proposition}
\label{pr: cocycle} Let $\alpha = \gamma_a \ast \gamma_1 \ast \gamma_2 \in
Z^2_H(A,K)$. Then $A^{\alpha}$ is isomorphic to the lifting $A(a_1,a_2,a)$
of $A$ described in Proposition \ref{pr: lift}.
\end{proposition}

\begin{proof}
We must describe the multiplication $\cdot_{\alpha}$ in the Hopf algebra $%
A^{\alpha}$. Note that $x_i^n \cdot_\alpha x_i^m = x_i^{n+m}$ for $n+m <r$
since each of the cocycles $\gamma_i$ and $\gamma_a$, and their inverses,
are $0$ on $x_i^k \otimes x_i^l$ when $j+l<r$. If $n+m = r$ then
\begin{equation*}
x_i^n \cdot_\alpha x_i^m = \alpha(g_i^n \otimes g_i^m)x_i^r \alpha^{-1}(1
\otimes 1) + \alpha(x_i^n \otimes x_i^m) + g_i^r \alpha^{-1}(x_i^n \otimes
x^m_i) = a_i(1 - g_i^r).
\end{equation*}
Now note that by the definition of $\alpha$, then $x_1^n \cdot_\alpha x_2^m
= x_1^n x_2^m$. However
\begin{equation*}
x_2 \cdot_\alpha x_1 = \alpha(x_2 \otimes x_1) + x_2x_1 + g_2g_1
\alpha^{-1}(x_2 \otimes x_1) = q x_1 \cdot_\alpha x_2 + a(1 - g_2g_1)
\end{equation*}
as required. Since multiplication is associative, this completes the proof.
\end{proof}

We summarize the action of the cocycle $\alpha$ on $A \otimes A$. For $0 <
i,k,m,n,t <r$ we have

\begin{itemize}
\item[(i)] $\alpha(z \otimes 1) = \alpha(1 \otimes z) = \varepsilon(z)$ for
all $z \in A$.

\item[(ii)] $\alpha(x_i^n \otimes x_i^m ) = \delta_{n+m,r} a_i .$

\item[(iii)] $\alpha (x_1^m \otimes x_2^k) = 0.$

\item[(iv)] $\alpha(x_1^i \otimes x_1^k x_2^m) = 0 = \alpha(x_1^ix_2^k
\otimes x_2^m).$

\item[(v)] $\alpha(x_2^m \otimes x_1^n) = \delta_{n,m} (m)!_qa^m.$

\item[(vi)] $\alpha(x_2^i \otimes x_1^kx_2^m) = \delta_{i+m,r+k}\binom{i}{k}%
_q(k)!_qa^ka_2.$

\item[(vii)] $\alpha (x_1^ix_2^k \otimes x_1^m) = \delta_{i+m,r+k}\binom{m}{k%
}_q(k)!_qa^ka_1.$

\item[(viii)] $\alpha (x_1^ix_2^k \otimes x_1^m x_2^t) =
\delta_{i+m,k+t}(i+m-r)!_q\binom{k}{r-t}_q \binom{m}{r-i}_q
q^{it}a^{i+m-r}a_1a_2.$
\end{itemize}

\begin{example}
\label{ex: dim81} Let us describe $\alpha$ completely for the Hopf algebras
of dimension $81$ which were among the first counterexamples to Kaplansky's
Tenth Conjecture. Here $\Gamma = \langle c \rangle$ is the cyclic group of
order $9$, $g_1 = c = g_2$, $\chi(c)=q$ where $q$ is a primitive cube root
of $1$, $r=3$. By \cite{Mas1}, there exists a cocycle $\alpha$ such that $%
A(a_1,a_2,a) \cong A^\alpha$. Here we supply $\gamma$ explicitly for $a_i$
and $a$ nonzero.

From the preceding computations, we see that $\alpha = \varepsilon$ except
for the following cases:
\begin{eqnarray*}
&& \alpha(x_2 \otimes x_1) = (1)!_qa^1 = a; \\
&& \alpha(x_i \otimes x_i^2)= \alpha(x_i^2 \otimes x_i) = a_i \text{ for }i
= 1,2; \\
&& \alpha(x_2^2 \otimes x_1^2) = (2)!_qa^2 = (1+q)a^2. \\
&& \alpha(x_2^2 \otimes x_1x_2^2) = {\binom{{2}}{{1}}}_q (1)!_q aa_2 =
(2)!_qaa_2 = (1 + q)aa_2; \\
&& \alpha(x_1^2x_2 \otimes x_1^2) = {\binom{{2}}{{1}}}_q (1)!_q aa_1 =
(2)!_qaa_1 = (1 + q)aa_1. \\
&& \alpha(x_1^2x_2^2 \otimes x_1x_2) = (0)!_q {\binom{{2}}{ 2} }_q {\binom{{1%
}}{{1}}}_q q^2 a^0a_1a_2 = q^2a_1a_2 =-(1+q)a_1a_2; \\
&& \alpha(x_1^2x_2 \otimes x_1x_2^2) = (0)!_q {\binom{{1}}{ 1} }_q {\binom{{1%
}}{{1}}}_q q^4 a^0a_1a_2 = q a_1a_2; \\
&& \alpha(x_1x_2^2 \otimes x_1^2x_2) = (0)!_q {\binom{{2}}{ 2} }_q {\binom{{2%
}}{{2}}}_q q^{1} a^0a_1a_2 = qa_1a_2; \\
&& \alpha(x_1 x_2 \otimes x_1^2x_2^2) = (0)!_q {\binom{{1}}{ 1} }_q {\binom{{%
2}}{{2}}}_q q^2 a^0a_1a_2 = q^2a_1a_2 =-(1+q)a_1a_2. \\
&& \alpha(x_1^2x_2^2 \otimes x_1^2x_2^2) = (1)!_q {\binom{{2}}{{1}}}_q {\
\binom{{2}}{{1}}}_q q^4 a a_1a_2 = (2)!_q (2)!_q qaa_1a_2 \\
&& \hspace{3cm} = (1+q)^2qaa_1a_2 =-(1+q)aa_1a_2.
\end{eqnarray*}

In the last case, $i + m = k+t = 4$, and in the four preceding cases, $%
i+m=k+t = r=3$. \qed
\end{example}

We ask whether, in the example above, it is possible to find $\eta$ such
that $\alpha = e^\eta = \varepsilon + \eta + \frac{\eta^{2}}{2!}+\frac{%
\eta^{3}}{3!}+\cdots$. This is the approach of \cite{Grunenfelder-Mastnak}
to construct cocycles where it may be simpler to construct $\eta$. Then one
expects that
\begin{equation*}
\eta=ln\alpha=ln\left(\varepsilon+\left(\alpha-\varepsilon\right)\right)=
\left(\alpha -\varepsilon\right)-\frac{\left(\alpha-\varepsilon\right)^{2}}{2%
}+ \frac{\left(\alpha-\varepsilon\right)^{3}}{3}-\cdots
\end{equation*}
and one checks (using a computer algebra system) that $(\alpha-\varepsilon
)^{3}=0 $ so that $\eta=ln\alpha=ln\left(\varepsilon+\left(\alpha-%
\varepsilon\right)\right)= \left(\alpha -\varepsilon\right)-\frac{%
\left(\alpha-\varepsilon\right)^{2}}{2}$. The map $\eta$ is explicitly given
by the table:

\begin{equation*}
\eta=\left[%
\begin{array}{c|c|c|c|c|c|c|c|c|c}
\eta(u\otimes v) & v=1 & x_{1} & x_{2} & x_{1}x_{2} & x_{1}^{2} & x_{2}^{2}
& x_{1}x_{2}^{2} & x_{1}^{2}x_{2} & x_{1}^{2}x_{2}^{2} \\ \hline
u=1 & 0 & 0 & 0 & 0 & 0 & 0 & 0 & 0 & 0 \\ \hline
x_{1} & 0 & 0 & 0 & 0 & a_{1} & 0 & 0 & 0 & 0 \\ \hline
x_{2} & 0 & a & 0 & 0 & 0 & a_{2} & 0 & 0 & 0 \\ \hline
x_{1}x_{2} & 0 & 0 & 0 & 0 & 0 & 0 & 0 & 0 & 0 \\ \hline
x_{1}^{2} & 0 & a_{1} & 0 & 0 & 0 & 0 & 0 & 0 & 0 \\ \hline
x_{2}^{2} & 0 & 0 & a_{2} & 0 & (1+\frac{q}{2})a^{2} & 0 & (\frac{1}{2}%
+q)aa_{2} & 0 & 0 \\ \hline
x_{1}x_{2}^{2} & 0 & 0 & 0 & 0 & 0 & 0 & 0 & 0 & 0 \\ \hline
x_{1}^{2}x_{2} & 0 & 0 & 0 & 0 & (\frac{1}{2}+q)aa_{1} & 0 & 0 & 0 & 0 \\
\hline
x_{1}^{2}x_{2}^{2} & 0 & 0 & 0 & 0 & 0 & 0 & 0 & 0 & -\frac{1}{2}aa_{1}a_{2}%
\end{array}%
\right]
\end{equation*}

\vspace{2.5mm} and $e^\eta = \alpha$.

\subsection{Pointed Hopf algebras of dimension 32}

In the next example, we study three infinite families of pointed Hopf
algebras of dimension $32$. Let $\Gamma$ be a finite abelian group of order $%
8$ and let $V \in {^\Gamma_\Gamma \mathcal{YD}}$ be a quantum linear space
as in Definition \ref{de: qls} with $r_i = 2$, $i=1,2$. For each of the $%
\Gamma, V$ listed below, the families of pointed Hopf algebras obtained by
lifting the Radford biproduct yield infinite families of non-isomorphic Hopf
algebras \cite[Section 5]{grana}; family $(F2)$ is also mentioned in \cite%
{beattieiso}. We have shown by the above explicit computation of the cocycle
that the pointed Hopf algebras in each of the families of liftings are
isomorphic to a twisting of the Radford biproduct. The same result, without
an explicit formula for the cocycle, can also be found as follows. In view
of \cite[Theorem 3.1]{eg} each element in a family $(Fi)$ is of the form $%
A(G_i,V_i,u_i,B)^*$ for some datum $(G_i,V_i,u_i,B)$. The construction of
the Hopf algebra $A(G_i,V_i,u_i,B)$ can be found at the beginning of \cite[%
Section 2]{eg}: it can be obtained by applying \cite[Theorem 3.1.1]{AEG} to
the datum $((\mathbb{C}[G_i]\ltimes\wedge V_i)^{e^B},u_i)$. Now, by
construction, for each $B,B^{\prime}$ the Hopf superalgebras $(\mathbb{C}%
[G_i]\ltimes\wedge V_i)^{e^B}$ and $(\mathbb{C}[G_i]\ltimes\wedge
V_i)^{e^{B^{\prime}}}$ are twist equivalent by twisting the
comultiplication. Thus, by \cite[Proposition 3.2.1]{AEG}, $A(G_i,V_i,u_i,B)$
and $A(G_i,V_i,u_i,B^{\prime})$ are also twist equivalent by twisting the
comultiplication. Thus their duals are quasi-isomorphic in our sense.

Etingof and Gelaki have shown in \cite[Corollary 4.3]{eg} that the families
of duals contain infinitely many quasi-isomorphism classes of Hopf algebras.
\vspace{1mm}

The three families are the pointed Hopf algebras of liftings of biproducts {%
\ corresponding to $\Gamma, V$ as follows: }

\begin{enumerate}
\item[$(F1)$] $\Gamma = C_8 = \langle g \rangle$, the cyclic group of order
8 with generator $g$ and $\eta \in \widehat{\Gamma}$ with $\eta(g) = q$, $q$
a primitive $8$th root of unity. $V = Kx_1 \oplus Kx_2$ where $x_1 \in V^{
\eta^4}_g$ and $x_2 \in V^{ \eta^4}_{g^5}$.

\item[$(F2)$] $\Gamma = C_8 = \langle g\rangle$, and $\eta \in \hat{\Gamma}$
as above. $V = Kx_1 \oplus Kx_2$ where $x_1 \in V^{ \eta^4}_g$ and $x_2 \in
V^{ \eta^4}_{g^3}$.

\item[$(F3)$] $\Gamma = C_2 \times C_4$ where $C_2 = \langle g \rangle$ and $%
C_4 = \langle h \rangle$. Let $\eta \in \hat{\Gamma}$ be defined by $\eta(g)
= 1$, $\eta(h) = q$ where $q$ is a primitive $4^{th}$ root of unity. $V =
Kx_1 \oplus Kx_2$ where $x_1 \in V^{ \eta^2}_h$ and $x_2 \in V^{
\eta^2}_{gh} $.
\end{enumerate}

\vspace{2mm}

\begin{example}
\label{ex: dim32} Let $A = \mathcal{B}(V) \# K[\Gamma]$ for any of the $%
V,\Gamma$ in $(F1)-(F3)$. Let $H$ denote $K[\Gamma]$ and let $a_1,a_2,a \in
K $. Define $\gamma = \gamma(a_1,a_2,a):A \otimes A \longrightarrow K$ by $%
\gamma(x_1^ix_2^j \otimes x_1^kx_2^l) = 0 $ if $i+j \neq k+l$,
\begin{eqnarray*}
&& \gamma (1 \otimes -)= \gamma( - \otimes 1) = \varepsilon; \\
&& \gamma( x_i \otimes x_i)= a_i; \hspace{2mm} \gamma (x_2 \otimes x_1) = a;
\hspace{2mm} \gamma(x_1 \otimes x_2) = 0; \\
&& \gamma (x_1x_2 \otimes x_1x_2) = -a_1a_2,
\end{eqnarray*}
and extend $\gamma$ to $A \otimes A$ by $H$-bilinearity. Then by Proposition %
\ref{pr: cocycle}, $\gamma \in Z^2_H(A ,K)$ and $A^\gamma \cong A(a_1,a_2,a)$%
.

We note that for $\gamma$ to be a cocycle we must have that $\gamma (x_1x_2
\otimes x_1x_2) = -a_1a_2$; to see this, apply (\ref{form: cocycle1 Hopf})
to the triple $x_1,x_2,x_1x_2$. \qed
\end{example}

\begin{remark}
In general, for $\gamma \in Z^2_H(A , K)$, the convolution inverse of $%
\gamma $ will be a cocycle for $A^\gamma$, but need not be a cocycle for $A$%
. In Example \ref{ex: dim32}, however, $\gamma^{-1}$ is also a cocycle for $%
A $ with the scalars $a_i,a$ replaced by their negatives.
\end{remark}

\begin{remarks}
\label{rem: 32}

(i) Let $B := A(a_1,a_2,a)$ be a lifting of the Radford biproduct $A =
\mathcal{B}(V) \# K[\Gamma]$ , from Example \ref{ex: dim32}. Then $B$ gives
a splitting datum $(B \cong R \#_\xi H, H , \pi_B, \sigma)$ where $\sigma$
is inclusion, $H = K[\Gamma] $ and $R = B^{co\pi_B}$ with the projection $%
\pi_B$ as described in Example \ref{ex: pbw}. From the definition of the
cocycle $\xi$ in \ref{sec: A co H}, we have that (taking $H$-bilinearity
into account) $\xi = \varepsilon \otimes \varepsilon$ except for the
following:
\begin{eqnarray*}
&& \xi(x_i \otimes x_i) = a_i(1-g_i^2); \hspace{2mm} \xi(x_2 \otimes x_1) =
a(1-g_1g_2 ); \\
&& \xi(x_1x_2 \otimes x_1x_2) = \pi(x_1(-x_1x_2 + a(1-g_1g_2))x_2 ) =
-a_1a_2(1-g_1^2)(1-g_2^2).
\end{eqnarray*}

By (\ref{form: xi inverse}), the inverse to $\xi$ is given by
\begin{eqnarray*}
&& \xi^{-1} = -\xi \text{ on } R_i \otimes R_j \text{ for } i+j <4 \text{
and } \xi^{-1}(x_1x_2 \otimes x_1x_2) = \xi(x_1x_2 \otimes x_1x_2).
\end{eqnarray*}
Since the image of the cocycle $\xi $ is in the centre of $B$, then $%
R=B^{co\pi_B}$ is associative. \vspace{1mm}

(ii) For the cocycle $\gamma$ defined in Example \ref{ex: dim32}, $R \#_\xi
H \cong \mathcal{B}(V)^\gamma \#_{\varepsilon_\gamma}H$, so that we have
splitting data $(R \#_\xi H , H, \pi_B,\sigma)$ and $( \mathcal{B}(V)^\gamma
\#_{\varepsilon_\gamma}H, H, \pi_A, \sigma)$. Since $\gamma^{\pm 1}(x_1
\otimes x_2) = 0$, then $m_A(x_1 \otimes x_2) = m_{A^\gamma}(x_1 \otimes
x_2) $ and $\pi_A = \pi_B$, $\mathcal{B}(V)^\gamma = R$, and $\xi =
\varepsilon_\gamma =\gamma_R \ast (H \otimes \gamma_R^{-1})\rho_{R \otimes
R} $. \vspace{1mm}

Let $\Lambda$ denote the total integral from $H$ to $K$; $\Lambda(g) =
\delta_{g,1}$. Then
\begin{equation*}
\Lambda \circ \xi = \Lambda \circ \varepsilon_\gamma = \Lambda \circ
(\gamma_R \ast (H \otimes \gamma^{-1}) \rho_{R \otimes R}) = (\gamma_R \ast
(\Lambda \otimes \gamma^{-1}) \rho_{R \otimes R}).
\end{equation*}
In family $(F1)$, $\rho_{R \otimes R}(z) = 1 \otimes z $ only when $z = 1
\otimes 1 $. Thus here $\Lambda \circ \pi \circ m_B = \Lambda \circ \xi =
\gamma_R$.

However in families $(F2)$ and $(F3)$, we have $\rho_{R \otimes R}(z) = 1
\otimes z $ for $z = 1 \otimes 1$ and also for $z = x_1x_2 \otimes x_1 x_2$
so that
\begin{equation*}
\Lambda \circ \xi (x_1x_2 \otimes x_1 x_2) = \gamma_R(x_1x_2 \otimes x_1
x_2) + \gamma_R^{-1}(x_1x_2 \otimes x_1 x_2) = 2 \gamma_R(x_1x_2 \otimes x_1
x_2).
\end{equation*}
One can also see this directly by applying $\Lambda$ to $\xi$ as described
above.

(iii) In Example \ref{ex: dim81}, $( \Lambda \otimes \gamma^{-1}) \circ
\rho_{R \otimes R}$ is nonzero only on $1 \otimes 1 $ since $\rho_{R \otimes
R}(x_1^ix_2^j \otimes x_1^kx_2^m) = c^{i+j+k+m} \otimes x_1^ix_2^j \otimes
x_1^kx_2^m$ and $i + j + k + m \leq 8$. Thus $\Lambda \circ \xi = \gamma$. %
\qed
\end{remarks}

Even though $\Lambda \circ \xi$ might not be a cocycle for $R$ above, $%
\Lambda \circ \xi$ is still a left $H$-linear map since $\Lambda$ is
ad-invariant and $\xi$ is left $H$-linear with respect to the adjoint
action. The next lemma will apply to this situation. Coalgebras and the
braiding in the category ${_H^H{\mathcal{YD}}} $ are described in Section %
\ref{sec: ydcoalgebras}.

\begin{lemma}
\label{lm: cocommutative} For $H$ a Hopf algebra, let $R$ be a coalgebra in $%
{_H^H{\mathcal{YD}}}$. Let $c$ be the braiding in ${_H^H{\mathcal{YD}}} $.
Suppose that $x,y$ are elements of $R$ such that

\begin{itemize}
\item[(i)] $c \Delta_R (x)= \Delta_R(x)$ and $c \Delta_R (y)= \Delta_R(y)$,

\item[(ii)] $(R \otimes c^2 \otimes R )(\Delta_R \otimes \Delta_R)(x \otimes
y) = (\Delta_R \otimes \Delta_R)(x \otimes y)$.
\end{itemize}

Let $\omega \in Hom(R \otimes R ,K)$ be left $H$-linear and let $%
\mu:R\otimes R\rightarrow R$ be a linear map. Then
\begin{equation*}
(\omega \ast \mu)( x \otimes y ) = (\mu \ast \omega)( x \otimes y ).
\end{equation*}
\end{lemma}

\begin{proof}
Set $z:=x \otimes y$. Recall that $c_{R\otimes R,R\otimes R}=\left( R\otimes
c\otimes R\right) \left( c\otimes c\right) \left( R\otimes c\otimes R\right)
$ and thus
\begin{eqnarray*}
z_{\left\langle -1\right\rangle }^{\left( 1\right) }\cdot z^{\left( 2\right)
}\otimes z_{\left\langle 0\right\rangle }^{\left( 1\right) } &=&c_{R\otimes
R,R\otimes R} ( z^{\left( 1\right) }\otimes z^{\left( 2\right) } ) \\
&=&\left( R\otimes c\otimes R\right) \left( c\otimes c\right) \left(
R\otimes c\otimes R\right) ( z^{\left( 1\right) }\otimes z^{\left( 2\right)
} ) \\
&=&\left( R\otimes c\otimes R\right) \left( c\otimes c\right) \left(
R\otimes c\otimes R\right) \left( R\otimes c\otimes R\right) \left( \Delta
_{R}\left( x\right) \otimes \Delta _{R}\left( y\right) \right) \\
&=&\left( R\otimes c\otimes R\right) \left( c\otimes c\right) \left( \Delta
_{R}\left( x\right) \otimes \Delta _{R}\left( y\right) \right) \\
&=&\left( R\otimes c\otimes R\right) \left( c\Delta _{R}\left( x\right)
\otimes c\Delta _{R}\left( y\right) \right) \\
&=&\left( R\otimes c\otimes R\right) \left( \Delta _{R}\left( x\right)
\otimes \Delta _{R}\left( y\right) \right) =z^{\left( 1\right) }\otimes
z^{\left( 2\right) }.
\end{eqnarray*}%
Hence%
\begin{eqnarray*}
\left( w\ast \mu \right) \left( z\right) &=&w ( z^{\left( 1\right) } )\mu (
z^{\left( 2\right) } ) =w ( z_{\left\langle -1\right\rangle }^{\left(
1\right) }\cdot z^{ ( 2 ) } ) \mu ( z_{\left\langle 0\right\rangle }^{\left(
1\right) } ) \\
&=&\varepsilon _{H} ( z_{\left\langle -1\right\rangle }^{ ( 1 ) } ) w (
z^{\left( 2\right) } ) \mu ( z_{\left\langle 0\right\rangle }^{\left(
1\right) } ) =w ( z^{\left( 2\right) } ) \mu ( z^{\left( 1\right) } ) \\
&=& \left( \mu \ast w\right) \left( z\right).
\end{eqnarray*}
\end{proof}

For example, if $V= Kx \oplus Ky$ is a quantum linear plane and $R
= \mathcal{B}(V)$, then the conditions of Lemma \ref{lm:
cocommutative} apply to $x,y$ with $\mu = m_R$. In the examples of
dimension $32$ in this section, $c^2$ is the identity on $R \otimes R$ and $%
c \Delta_R = \Delta_R$. If $\omega$ is convolution invertible and left $H$%
-linear, then
\begin{equation*}
\omega \ast m_R \ast \omega^{-1} = m_R.
\end{equation*}
In particular, $\omega$ could be $\Lambda \circ \xi$. \vspace{3mm}

\subsection{Some general remarks}

Given a general splitting datum $(A = R \#_\xi H,H,\pi,\sigma)$, one problem
is to find $\omega \in Z^2_H(A,K)$ such that $(A^\omega,H,\pi,\sigma)$ is a
trivial splitting datum, in other words, such that $\xi_{\omega_R}$ is
trivial.

As in Remarks \ref{rem: 32}, from the definition in Theorem \ref{teo: smash
gamma}, if $\xi_{\omega_R} = \varepsilon$, then:
\begin{eqnarray*}
&&\xi=u_{H}\omega_{R}^{-1}\ast \left( H\otimes \omega_{R}\right) \rho
_{R\otimes R} = u_{H}\omega_{R}^{-1}\ast \Psi(\omega_R),
\end{eqnarray*}
and then for any $f \in \mathrm{{Hom}(H,K)}$,
\begin{equation*}
f \circ \xi = \omega_R^{-1} \ast (f \otimes \omega_R)(\rho_{R \otimes R}),
\end{equation*}
so that
\begin{equation}  \label{eq: lambdaxi}
f \circ \xi =\omega_{R}^{-1} \text{ if and only if } \left( f \otimes
\omega_{R}\right) \rho _{R\otimes R}=\varepsilon _{R}\otimes \varepsilon
_{R}.
\end{equation}

Similarly
\begin{equation}  \label{eq: lambdaxi-1}
f \circ \xi ^{-1}=\omega_{R} \text{ if and only if } \left( f \otimes
\omega_{R}^{-1}\right) \rho _{R\otimes R}=\varepsilon _{R}\otimes
\varepsilon _{R}.
\end{equation}

Even though we know that that $\Psi(\omega_R^{-1}) = \left( H\otimes
\omega_{R}^{-1}\right) \rho _{R\otimes R} $ is the convolution inverse to $%
\Psi(\omega_R ) = \left( H\otimes \omega_{R}\right) \rho _{R\otimes R}$, it
is not clear if the equalities in (\ref{eq: lambdaxi}) and (\ref{eq:
lambdaxi-1}) are equivalent.

\vspace{2mm} If $f$ above is an integral $\lambda$ for $H$, we know from the
examples of dimension 32 that $\lambda \circ \xi$ is not always a cocycle.
Nevertheless, if it is a cocycle, then twisting by $\lambda \circ \xi$
yields a trivial splitting datum.

\begin{proposition}
Let $(A,H,\pi,\sigma)$ be a splitting datum with associated pre-bialgebra $%
(R,\xi)$. Let $\lambda \ $ be a left integral for $H$ in $H^{\ast }.$ Then%
\begin{equation*}
\xi \ast \left[ \left( H\otimes \lambda \xi \right) \rho _{R\otimes R}\right]
=u_{H}\lambda \circ \xi .
\end{equation*}%
For $A = R \#_\xi H$, if $\gamma \in Z^2_H(A,K)$ such that $\lambda \circ
\xi =\gamma _{R}^{-1},$ then $\xi_{\gamma_R}$ is trivial and $\left(
R\#_{\xi }H\right) ^{\gamma } = R^{\gamma _{R}}\#H.$
\end{proposition}

\begin{proof}
We have%
\begin{equation*}
\xi \ast \left[ \left( H\otimes \lambda \xi \right) \rho _{R\otimes R}\right]
=(m_{H}\otimes \lambda \xi )(\xi \otimes \rho _{R\otimes R})\Delta
_{R\otimes R}\overset{\left( \ref{eq: Sweedler 1-cocycle}\right) }{=}\left(
H\otimes \lambda \right) \Delta _{H}\xi =u_{H}\lambda \xi .
\end{equation*}%
If $\gamma \in Z^2_H(A,K)$ such that $\lambda \xi =\gamma _{R}^{-1}$, then $%
\xi \ast \left( H\otimes \gamma _{R}^{-1}\right) \rho _{R\otimes R}=\xi \ast
\left( H\otimes \lambda \xi \right) \rho _{R\otimes R}=u_{H}\lambda \xi
=u_{H}\gamma _{R}^{-1}$ so that
\begin{equation*}
\xi_{\gamma_R}=u_{H}\gamma _{R}\ast \xi \ast \left( H\otimes \gamma
_{R}^{-1}\right) \rho _{R\otimes R}=u_{H}\gamma _{R}\ast u_{H}\gamma
_{R}^{-1}=u_{H}\varepsilon _{R\otimes R}
\end{equation*}
\end{proof}

The properties of $\lambda\xi$ will be investigated in a forthcoming paper.
\vspace{2mm}

\newpage

\begin{center}
\ Acknowledgement: \\[0pt]
\textit{Our thanks to the referee for his or her careful reading
of this paper.}
\end{center}

\end{document}